\documentclass{article}
\usepackage{amsthm}
\usepackage{amssymb}
\usepackage{amsfonts}
\usepackage{amsmath}
\begin{document}

\newtheorem{theorem}{Theorem}[section]
\newtheorem{definition}{Definition}[section]
\newtheorem{corollary}[theorem]{Corollary}
\newtheorem{lemma}[theorem]{Lemma}
\newtheorem{proposition}[theorem]{Proposition}
\newtheorem{step}[theorem]{Step}
\newtheorem{example}[theorem]{Example}
\newtheorem{remark}[theorem]{Remark}

\font\sixbb=msbm6
\font\eightbb=msbm8
\font\twelvebb=msbm10 scaled 1095
\newfam\bbfam
\textfont\bbfam=\twelvebb \scriptfont\bbfam=\eightbb
                           \scriptscriptfont\bbfam=\sixbb

\newcommand{\tr}{{\rm tr \,}}
\newcommand{\linspan}{{\rm span\,}}
\newcommand{\rank}{{\rm rank\,}}
\newcommand{\diag}{{\rm Diag\,}}
\newcommand{\Image}{{\rm Im\,}}
\newcommand{\Ker}{{\rm Ker\,}}

\def\bb{\fam\bbfam\twelvebb}
\newcommand{\enp}{\begin{flushright} $\Box$ \end{flushright}}

\def\cB{{\mathcal{B}}}
\def\cD{{\mathcal{D}}}
\def\cU{{\mathcal{U}}}
\def\cV{{\mathcal{V}}}

\def\C{{\mathbb{C}}}
\def\D{{\mathbb{D}}}
\def\N{{\mathbb{N}}}
\def\Q{{\mathbb{Q}}}
\def\R{{\mathbb{R}}}
\def\Z{{\mathbb{Z}}}

\title{Loewner's theorem for maps on operator domains\thanks{The first author was supported by Leading Graduate Course for Frontiers of Mathematical Sciences and Physics (FMSP), JSPS KAKENHI 19J14689 and JSPS Overseas Challenge Program, Japan. The second author was supported by grants N1-0061, J1-8133, and P1-0288 from ARRS, Slovenia.}}
\author{Michiya Mori\footnote{Graduate School of Mathematical Sciences, The University of Tokyo, 3-8-1 Komaba Meguro-ku Tokyo, 153-8914, Japan, mmori@ms.u-tokyo.ac.jp} \ \, and \ \, 
Peter \v Semrl\footnote{Faculty of Mathematics and Physics, University of Ljubljana,
        Jadranska 19, SI-1000 Ljubljana, Slovenia; Institute of Mathematics, Physics, and Mechanics, Jadranska 19, SI-1000 Ljubljana, Slovenia, peter.semrl@fmf.uni-lj.si}
        }

\date{}
\maketitle

\begin{abstract}
The classical Loewner's theorem states that operator monotone functions on real intervals are described by holomorphic functions on the upper half-plane. 
We characterize local order isomorphisms on operator domains by biholomorphic  automorphisms of the generalized upper half-plane, which is the collection of all operators with positive invertible imaginary part. 
We describe such maps in an explicit manner, and examine properties of maximal local order isomorphisms. 
Moreover, in the finite-dimensional case, we prove that every order embedding of a matrix domain is a homeomorphic order isomorphism onto another matrix domain. 
\end{abstract}
\maketitle




\section{Introduction}
A real function $f$ defined on an open interval $(a, b)$ is said to be matrix monotone of order $n$ if for every pair of $n\times n$ hermitian matrices $X, Y$ whose eigenvalues belong to $(a, b)$ we have $X\le Y \Rightarrow f(X)\le f(Y)$. 
If $f$ is a matrix monotone function of order $n$ for all positive integers $n$, we say that $f$ is operator monotone. 
The study of operator monotone functions was initiated by Loewner \cite{Low}. 
His famous theorem states that a function $f : (a, b) \to {\R}$ is operator monotone if and only if 
$f$ has an analytic continuation to the upper half-plane $\Pi$ which maps $\Pi$ into itself, and this is true if and only if
$f$ has an analytic extension  to $(\mathbb{C} \setminus \mathbb{R}) \cup (a,b)$ which maps
the upper half-plane $\Pi$ into itself, and the extension to the lower half-plane is obtained by the reflection across the real line.
The theory of operator monotone functions has found many important applications, see for example \cite{KuA, WiN}. 
Over the years enormous interest has been paid to this deep theorem. 
Recently, a multivariable version has been obtained \cite{AMY}. 
See Donoghue's book \cite{Don} for further information about the classical Loewner's theorem.
A number of alternative proofs can be found in Simon's book \cite{Sim}. 
Furthermore, a review of this book \cite{McC} can serve as a quick gentle introduction to this topic. 

We are going to consider a much more general setting. 
Let $H$ be a complex Hilbert space. 
To avoid trivialities, we assume throughout the paper that $H$ has dimension at least $2$. 
We denote by $B(H)$ the algebra of all bounded linear operators on $H$, and by $S(H)$ the subset of all self-adjoint operators.
Let $U,V$ be subsets of $S(H)$. A map $\phi : U \to V$ preservers order (in one direction) if for every pair $X, Y\in U$ we have $X\le Y \Rightarrow \phi(X)\le \phi(Y)$, and it is an order embedding (or preservers order in both directions) if for every pair $X, Y\in U$ we have $X\le Y \iff \phi(X)\le \phi(Y)$. It is easy to see that an order embedding is injective.
If $\phi$ is a bijective order embedding, it is called an order isomorphism.  

A nonempty subset $U \subset S(H)$ will be called an operator domain if it is open and connected. Here and throughout the paper, the topology on $B(H)$ is induced by the operator norm.
Let $U$ be an operator domain. A map $\phi : U\to S(H)$ is defined to be a local order isomorphism if for every $X\in U$ there are operator domains $V, W$ with  $X\in V\subset U, W\subset S(H)$, such that $\phi(V)=W$ and $\phi : V\to W$ is an order isomorphism.
The generalized upper half-plane $\Pi(H)$ is the collection of all operators of the form $X+iY$, where $X\in S(H)$ and $Y$ is a positive invertible operator in $S(H)$. The generalized lower half-plane $\Pi(H)^\ast = \{ X^\ast \, : \, X \in \Pi (H) \}$ is the set of all bounded operators on $H$ whose imaginary part is negative and invertible.

We are now ready to formulate our main theorem.
\begin{theorem}\label{punodosta}
Let $U  \subset S(H)$ be an operator domain. The following conditions are equivalent for a map $\phi : U \to S(H)$.
\begin{itemize}
\item The map $\phi$ is a local order isomorphism.
\item The map $\phi$ has a unique continuous extension to $U \cup \Pi (H)$ that maps $\Pi (H)$ biholomorphically onto itself.
\item There exist open connected sets $\cU, \cV\subset B(H)$ such that $U \cup \Pi(H)  \cup \Pi(H)^\ast \subset \cU$ and $\phi$ has an extension to a biholomorphic map from $\cU$ onto $\cV$ that maps  $\Pi(H)$ onto itself and $\Pi(H)^\ast$ onto itself.
\end{itemize}
\end{theorem}

Biholomorphy in our statement replaces holomorphy in the classical Loewner theory, and local order isomorphisms appear in the place of operator monotone functions. 
This is inevitable.
It is not difficult to construct order preserving maps on operator domains that do not have holomorphic extensions to the generalized upper half-plane even under the additional assumptions of bijectivity and continuity. 
Local operator monotonicity appears already in the scalar case, see \cite[Theorem 6.2]{Sim}. 

We have a much better result in the matrix case.
We denote by $M_n$ the set of all $n\times n$ complex matrices and by $S_n$ the set of all $n\times n$ hermitian matrices. 
The corresponding generalized upper half-plane $\Pi_n$ is the collection of all $n\times n$ complex matrices whose imaginary part is a positive invertible matrix. To avoid unnecessary repetition we present a version with only the generalized upper half-plane involved.
\begin{theorem}\label{coffee}
Let $n\ge 2$ and $U\subset S_n$ be a matrix domain.
A map $\phi : U\to S_n$ is an order embedding if and only if $\phi$ has a unique continuous extension to $U \cup \Pi_n$ that maps $\Pi_n$ biholomorphically onto itself.
\end{theorem}
Compare our results with \cite{PTD}, in which Pascoe and Tully-Doyle studied a ``free'' analogue of the classical Loewner's theorem, and discovered the relation between order preserving maps and holomorphic maps of the generalized upper half-plane under the assumption of freeness.

In the second section we will prove that the third condition in the main theorem implies the first one. We will start with maps that are not biholomorphic but only holomorphic on the generalized upper half-plane and verify that this assumption is enough to yield the monotonicity on line segments.
The next section will be devoted to the proof that every local order isomorphism has a unique continuous extension that restricts to a biholomorphic automorphism of the generalized upper half-plane. We will continue by giving explicit formulae 
for biholomorphic automorphisms of the generalized upper half-plane and local order isomorphisms. As a byproduct we will complete the proof of the main theorem.
In order to understand local order isomorphisms, we only need to study maximal ones, which we classify by a natural equivalence relation. 
We also give a necessary and sufficient condition for a maximal local order isomorphism to be an order isomorphism.
The crucial step in the proof of Theorem \ref{coffee} is an automatic continuity result for order embeddings of matrix domains. 
An explicit matrix formula for maximal order embeddings will be given. 
In the last section, we will present examples illustrating the optimality of our results, and show that they unify and significantly extend all known results on order isomorphisms of operator intervals, see \cite{Se6} and the references therein.

\section{From biholomorphic maps to local order isomorphisms}
We recall two definitions.
Let $U\subset \cB$ be an open subset of a real Banach space. 
A map $\phi: U\to \cB$ is Gateaux differentiable if for any $X\in U$ there exists a map $D\phi(X): \cB\to \cB$ such that 
$$
\left\|\frac{\phi(X+\varepsilon Y) - \phi(X)}{\varepsilon} - (D\phi(X))(Y)\right\| \to 0
$$
as $\R\ni \varepsilon \to 0$ for every $Y\in \cB$.

Let $\cU\subset \cB$ be an open subset of a complex Banach space. 
A map $\phi: U\to \cB$ is holomorphic or Fr\'echet differentiable if for any $X\in U$ there exists a bounded complex-linear map $D\phi(X): \cB\to \cB$ such that 
$$
\frac{\|\phi(X+Y) - \phi(X) - (D\phi(X))(Y)\|}{\| Y\|} \to 0
$$
as $\| Y \|\to 0$.
Let $\cU, \cV\subset \cB$ be open subsets of a complex Banach space.
A map $\phi: \cU\to \cV$ is said to be biholomorphic if $\phi$ is bijective and both $\phi$ and $\phi^{-1}$ are holomorphic.  

For two operators $A, B\in S(H)$, we write $A<B$ if $A\le B$ and $B-A$ is invertible in $B(H)$. 

\begin{lemma}\label{gateaux}
Let $U\subset S(H)$ be an operator domain. 
Suppose that $\phi: U\to S(H)$ is a Gateaux differentiable map with 
$(D\phi(X)) (A)\ge 0$ for any $X\in U$ and $A\in S(H)$ with $A\ge 0$. 
If $X, Y\in U$ satisfy $X\le Y$ and $(1-c)X+cY\in U$ for any real number $0\le c\le 1$, then $\phi(X)\le \phi(Y)$.
\end{lemma}
\begin{proof}
For a positive real number $a>0$, consider the Gateaux differentiable map $\phi_a : U\to S(H)$ defined by $\phi_a(X)=\phi(X)+aX$. 
Then we have $(D\phi_a(X)) (A)> 0$ for any $X\in U$ and $A\in S(H)$ with $A> 0$. 
It suffices to show the statement for $\phi_a$ instead of $\phi$ for every positive real number $a>0$. 

Let us assume $X<Y$ first. 
For each real number $c\in [0, 1)$, we have $(D\phi_a((1-c)X+cY))(Y-X)>0$. 
Since $(D\phi_a((1-c)X+cY))(Y-X)$ is equal to 
$$
\lim_{\varepsilon\to 0}
\frac{\phi_a((1-c)X+cY + \varepsilon (Y-X)) - \phi_a((1-c)X+cY)}{\varepsilon}  ,
$$
there exists a positive real number $ 0 <\varepsilon \le 1-c$ such that 
$\phi_a((1-c)X+cY + \varepsilon (Y-X)) > \phi_a((1-c)X+cY)$. 
In other words, if we define $X_c = (1-c)X + cY$, $c \in [0,1]$, then for every $c \in [0,1)$ there exists $d$ such that $c < d \le 1$ and $\phi_a (X_c) < \phi_a (X_d)$.
 Moreover, the Gateaux differentiability of $\phi_a$ implies that the map $[0, 1]\ni c\mapsto \phi_a(X_c)$ is continuous. 
Thus $\sup\{c\in (0, 1]\, : \,  \phi_a((1-c)X+cY)\ge \phi_a(X)\}=1$, so we obtain $\phi_a(X)\le\phi_a(Y)$. 

Consider the general case $X\le Y$. 
Since the collection $\{(1-c)X+cY\, : \, c\in [0, 1]\}$ is compact, there exists a positive real number $\varepsilon>0$ such that $(1-c)X+c(Y +d I)\in U$ for all $c\in [0, 1], d\in (0, \varepsilon]$. 
The preceding paragraph implies $\phi_a(X)\le \phi_a(Y+dI)$ for any $d\in (0, \varepsilon]$. 
Take the limit $d\to 0$ to obtain $\phi_a(X)\le \phi_a(Y)$. 
\end{proof}

In the last section we will present examples showing that when considering maps on operator domains we need to assume that order is (at least locally) preserved in both directions if we want to get any reasonable result. Then it is natural that in our analogue of the Loewner's theorem the holomorphy is replaced by the biholomorphy. Nevertheless, when dealing with the easier implication of our main theorem the situation is similar to the scalar case. That is, just the holomorphy is enough to deduce the monotonicity on line segments.

\begin{theorem}\label{1}
Let $\cU \subset B(H)$ be an open subset with $\cU\cap S(H)\neq \emptyset$ and $\Pi(H)\subset \cU$. Let $U \subset \cU\cap S(H)$ be an operator domain.  
Suppose that $\phi: \cU\to B(H)$ is a holomorphic map with $\phi(U)\subset S(H)$ and $\phi(\Pi(H))\subset \Pi(H)$. 
If $X, Y\in U$ satisfy $X\le Y$ and $(1-c)X+cY\in U$ for any real number $0\le c\le 1$, then $\phi(X)\le \phi(Y)$.
\end{theorem}
\begin{proof}
Let $X\in U$ and $A\in S(H)$, $A\ge 0$.
The assumption $\phi(U)\subset S(H)$ implies 
$$
(D\phi(X))(A) = \lim_{t \to 0} {1 \over  t} (\phi (X+ tA) - \phi (X))\in S(H).
$$ 
Since $\phi$ restricts to a holomorphic map of $\Pi(H)$ into itself, we have $(D\phi(X))(iA) = i (D\phi(X))(A)\in \overline{\Pi(H)}$, thus $(D\phi(X))(A)\ge 0$. 
It follows that $\phi$ restricted to $U$ satisfies the assumption of the preceding lemma. 
\end{proof}

We are now ready to prove one of the implications in our main result.

\begin{theorem}
Let $\cU, \cV\subset B(H)$ be open subsets with $\cU\cap S(H)\neq \emptyset$ and $\Pi(H) \cup \Pi (H)^\ast \subset \cU$. Let $U \subset \cU\cap S(H)$ be an operator domain.  
Suppose that $\phi: \cU\to \cV$ is a biholomorphic map that restricts to a biholomorphic automorphism of $\Pi(H)$ and also to a biholomorphic automorphism of $\Pi(H)^\ast$.   
Then $\phi$ restricts to a local order isomorphism of $U$. 
\end{theorem}
\begin{proof} Because $\phi: \cU\to \cV$ and $\phi^{-1} : \cV\to \cU$ are continuous bijections and $\phi (\Pi (H)) = \Pi (H)$ and $\phi (\Pi (H)^\ast) = \Pi (H)^\ast$ we have $\phi (\cU\cap S(H)) = \cV\cap S(H)$.
It follows that $\phi(U)$ is an operator domain in $S(H)$.
Let $X_0\in U$.
By the continuity of $\phi$, we may take a neighbourhood $X_0\in V\subset U$ with the following two properties: 
\begin{itemize}
\item For any pair $X, Y\in V$ with $X\le Y$ and any $c\in [0, 1]$, $(1-c)X+cY\in U$ holds.
\item For any pair $X, Y\in \phi(V)$ with $X\le Y$ and any $c\in [0, 1]$, $(1-c)X+cY\in \phi(U)$ holds.
\end{itemize}
Hence, by Theorem \ref{1}, $\phi$ restricts to an order isomorphism from $V$ onto $\phi(V)$.
\end{proof}

\section{From local order isomorphisms to biholomorphic maps}\label{hard}
We first prove  a kind of the identity theorem in the operator setting. 
\begin{proposition}\label{identity}
Let ${\cal U} \subset B(H)$ be an open connected set, $\emptyset\neq U\subset{\cal U}\cap S(H)$ an operator domain, and $\xi : {\cal U} \to B(H)$ a holomorphic map with $\xi(X)=0$ for every $X \in U$. 
Then $\xi(X)=0$ for all $X\in {\cal U}$.   
\end{proposition}
\begin{proof}
Fix $X\in U$. 
Take a positive real number $c>0$ such that $X+Y\in {\cal U}$ for any $Y\in B(H)$ with $\| Y \|\le 2c$, and $X+Y\in U$ for any $Y\in S(H)$ with $\| Y \|\le 2c$. 
Fix an operator $Y\in B(H)$ with $\| Y \|\le c$. 
We prove $\xi(X+Y)=0$. 
Take operators $Y_1, Y_2\in S(H)$ such that $Y=Y_1+iY_2$. 
Then $\| Y_1 \|, \| Y_2\|\le c$. 
For each $x, y\in H$, consider the holomorphic map $z \mapsto \langle \xi(X+Y_1+zY_2)x, y\rangle$ from some open neighbourhood of the closure of $\D:=\{z\in \C\, :\, |z|<1\}$ into $\C$.  
We have $X+Y_1+tY_2\in U$ for any $t\in [-1, 1]$, hence $\langle \xi(X+Y_1+tY_2)x, y\rangle =0$, $t\in [-1, 1]$.
The identity theorem implies $\langle \xi(X+Y_1+zY_2)x, y\rangle = 0$ for any $z\in \overline{\mathbb{D}}$. 
It follows $\xi(X+Y)=\xi(X+Y_1+iY_2)=0$. 

Consider the collection 
$$
{\cal U}_0 := \{X\in {\cal U}\, : \, \xi(Y)=0\text{ holds for every element }Y\text{ in some neighbourhood of }X\}.
$$
Then ${\cal U}_0\supset U$ and ${\cal U}_0$ is open. 
Let $X\in {\cal U}$. 
Take a continuous path $\tau : [0, 1]\to {\cal U}$ such that $\tau(0)\in {\cal U}_0$, $\tau(1)=X$. 
Then the collection $V=\{t\in [0, 1] \, : \, \tau(t)\in {\cal U}_0\}$ is open in $[0, 1]$. 
We prove $V$ is closed. 
Let $t\in [0, 1]$ be in the closure of $V$. 
Since ${\cal U}$ is open and $\tau$ is continuous, we may take numbers $s\in [0, 1]$ and $c>0$ such that $\tau(s)\in {\cal U}_0$, $\tau(s)+Y\in {\cal U}$ for any $Y\in B(H)$ with $\| Y\|\le c$ and $\|\tau(s)-\tau(t)\|<c$. 
Consider the holomorphic function $z\mapsto \langle\xi(\tau(s)+ zY)x, y\rangle$ on $\D$, for $Y\in B(H)$ with $\| Y\|\le c$ and  $x, y\in H$. 
Then the identity theorem implies  that this map is the constant $0$ map, and we have $t\in V$.
We obtain $X = \tau(1) \in {\cal U}_0$, which completes the proof. 
\end{proof}

For $A, B\in S(H)$ with $A\le B$, we define $[A, B] := \{X\in S(H)\, :\, A\le X\le B\}$. 
For $A\in S(H)$, we define  $[A, \infty) := \{X\in S(H)\, :\, A\le X\}$, $(A, \infty) := \{X\in S(H)\, :\, A< X\}$, $(-\infty, A] := \{X\in S(H)\, :\, X\le A\}$, and $(-\infty, A) := \{X\in S(H)\, :\, X< A\}$. 
For $A, B\in S(H)$ with $A< B$, we define $[A, B) := \{X\in S(H)\, :\, A\le X< B\}$, $(A, B] := \{X\in S(H)\, :\, A< X\le B\}$ and $(A, B) := \{X\in S(H)\, :\, A< X< B\}$.
Such collections are called operator intervals.

\begin{lemma}\label{[]}
Let $U\subset S(H)$ be an operator domain and $\phi : U\to S(H)$ be a local order isomorphism. 
Then for any $X\in U$ there exist operators $A, B\in U$ such that $A<X<B$, $[A, B]\subset U$, $\phi(A)<\phi(X)<\phi(B)$, $[\phi(A), \phi(B)]\subset \phi(U)$, $\phi([A, B]) = [\phi(A), \phi(B)]$, and $\phi : [A, B]\to [\phi(A), \phi(B)]$ is an order isomorphism.
\end{lemma}
\begin{proof}
By the definition of local order isomorphisms, it suffices to prove the statement with the additional assumption $\phi$ is an order isomorphism from $U$ onto the operator domain $\phi(U)$. 

Let $X, Y\in U$ satisfy $X<Y$. 
We prove $\phi(X)<\phi(Y)$. 
Since $\phi$ preserves order, we have $\phi(X)\le \phi(Y)$. 
Assume for a contradiction that $\phi(X)\not< \phi(Y)$. Then
we may find a vector $x$ in $H$ that is not in the range of $(\phi(Y)-\phi(X))^{1/2}$. 
Take the rank-one projection $P$ onto $\C x$. 
By Douglas' lemma, we have $\phi(X)+cP\nleq \phi(Y)$ for all positive real numbers $c>0$. 
Fix $c>0$ small enough. 
We have $\phi(X)+cP\in \phi(U)$, hence there exists an operator $Z\in U$ such that $\phi(Z)=\phi(X)+cP$. 
It follows that $\phi$ restricts to an order isomorphism from $[X, Y]\cap [X, Z]\cap U$ onto $[\phi(X), \phi(X)+cP]\cap [\phi(X), \phi(Y)]\cap \phi(U) =\{\phi(X)\}$, which is absurd. 
Applying the same argument to $\phi^{-1}$, we see that $X<Y$ if and only if $\phi(X)<\phi(Y)$.

Let $X\in U$. 
Since ${U}$ is open, there exists a positive real number $\varepsilon$ such that $[X-\varepsilon I, X+\varepsilon I]\subset {U}$. 
By the preceding paragraph, we have $\phi(X-\varepsilon I)<\phi(X)<\phi(X+\varepsilon I)$. 
Since $\phi({U})$ is open, there exists a positive real number $\varepsilon'$ such that $[\phi(X)-\varepsilon' I, \phi(X)+\varepsilon' I]\subset \phi({U})\cap [\phi(X-\varepsilon I), \phi(X+\varepsilon I)]$. 
It follows that $\phi^{-1}(\phi(X)-\varepsilon' I)<X< \phi^{-1}(\phi(X)+\varepsilon' I)$, $[\phi^{-1}(\phi(X)-\varepsilon' I),  \phi^{-1}(\phi(X)+\varepsilon' I)]\subset {U}$ and $\phi$ restricts to an order isomorphism from $[\phi^{-1}(\phi(X)-\varepsilon' I),  \phi^{-1}(\phi(X)+\varepsilon' I)]$ onto $[\phi(X)-\varepsilon' I, \phi(X)+\varepsilon' I]$.
\end{proof}

Let us recall a result by the second author on order isomorphisms between operator intervals. 
Let $f_p$, $p < 1$, denote a bijective monotone increasing function of the unit interval $[0,1]$ onto itself defined by
$$
f_p (x) = { x \over px + 1-p}, \ \ \ x \in [0,1].
$$
\begin{theorem}[{\cite{Se0, Se6}}]\label{fpfq}
Assume that $\phi : [0,I] \to [0,I]$ is an order automorphism.
Then there exist real numbers $p,q$, $0 <p < 1$, $q < 0$, and a bijective linear or conjugate-linear bounded operator
$T : H \to H$ with $\| T \| \le 1$ such that
\begin{equation}\label{qaui}
\phi (X) = f_q \left( \left( f_p (TT^* ) \right) ^{-1/2} f_p (TXT^*) \left( f_p (TT^* ) \right) ^{-1/2} \right), \ \ \
X \in [0,I].
\end{equation}
\end{theorem}

\begin{lemma}\label{ABCD}
Let $A, B, C, D\in S(H)$ be operators with $A<B$, $C<D$. 
Suppose that $\phi : [A, B]\to [C, D]$ is an order isomorphism. 
Then $\phi$ has an extension to a biholomorphic map on an open connected set $\cU\subset B(H)$ with $\Pi(H)\subset \cU$ that maps  ${\Pi} (H)$ onto itself.
\end{lemma}
\begin{proof}
We define order isomorphisms $\psi: [0, I]\to [A, B]$ and $\psi': [C, D]\to [0, I]$ by  
$\psi(X)=(B-A)^{1/2}X(B-A)^{1/2}+A$ and $\psi'(X)=(D-C)^{-1/2}(X-C)(D-C)^{-1/2}$.
Then $\psi'\circ\phi\circ\psi : [0, I]\to [0, I]$ is an order automorphism. 
It suffices to prove this lemma for $\psi'\circ\phi\circ\psi$ instead of $\phi$. 
Hence we may assume $[A, B]=[C, D]=[0, I]$. 

Let $\phi$ be an order automorphism of $[0, I]$.
By Theorem \ref{fpfq}, we obtain  
$$
\phi(X) =  f_q \left( \left( f_p (TT^\ast ) \right) ^{-1/2} f_p (TXT^\ast) \left( f_p (TT^\ast ) \right) ^{-1/2} \right), \ \ \
X \in [0,I].
$$
Note that this map can be decomposed into four order isomorphisms
\[
\begin{split}
\phi_1&: [0, I]\to [0, TT^\ast],\,\, X\mapsto TXT^*,\\
\phi_2&: [0, TT^\ast]\to [0, f_p(TT^*)],\,\, X\mapsto f_p(X),\\
\phi_3&: [0, f_p(TT^*)]\to [0, I],\,\, X\mapsto \left( f_p (TT^\ast ) \right) ^{-1/2} X \left( f_p (TT^\ast ) \right) ^{-1/2},\,\,\text{and}\\
\phi_4&: [0, I]\to [0, I],\,\, X\mapsto f_q(X).
\end{split}
\]
It suffices to prove the statement for $\phi_j$, $j=1, 2, 3, 4$. 
It is easy to see that $\phi_3$ continuously extends to a biholomorphic automorphism of $B(H)$ that maps $\Pi(H)$ onto itself.
The same holds for $\phi_1$  if $T$ is linear.
If $T$ is conjugate-linear, then the map $X\mapsto TX^*T^*$ on $B(H)$ extends $\phi_1$ and is a biholomorphic automorphism of $B(H)$ that maps $\Pi(H)$ onto itself.

Because for every $X+iY \in \Pi (H)$ (where $X, Y\in S(H)$) and every $Z \in S(H)$ we have
$$
(X+iY)^{-1} = Y^{-1/2}  ( Y^{-1/2} X  Y^{-1/2} + iI)^{-1} Y^{-1/2}
$$
and 
$$
(iI + Z)^{-1} = (I + Z^2)^{-1/2} (Z - iI)  (I + Z^2)^{-1/2},
$$
the map $X \mapsto -X^{-1}$ is a biholomorphic automorphism of $\Pi (H)$.
Let $r$ be a real number $<1$, $r\not=0$. Using
\begin{equation}\label{mimimi}
f_r (x) = {1 \over r} - {1 - r \over r^2} \left( {1-r \over r} + x \right)^{-1},
\end{equation}
and
$$
f_{r}^{-1} = f_{r \over r-1}
$$
we see that  $f_r$ extends to a biholomorphic map $X\mapsto f_r(X)$ from $\{X\in B(H)\, : \, 1- 1/r\notin\sigma(X)\}$ onto $\{X\in B(H)\, : \, 1/r\notin\sigma(X)\}$.
Since the map $X\mapsto -X^{-1}$ maps $\Pi(H)$ onto itself, we finally conclude that the map $X \mapsto f_r (X)$ is a biholomorphic automorphism of $\Pi (H)$.
\end{proof}

Let us prove that the first condition of the main theorem implies the second one.
\begin{proposition}\label{pikanogavicka}
Let $U\subset S(H)$ be an operator domain and $\phi : U\to S(H)$ be a local order isomorphism. 
Then $\phi$ has a unique continuous extension to  $U\cup \Pi (H)$ that maps  ${\Pi} (H)$ biholomorphically onto itself.
\end{proposition}
\begin{proof}
By Lemma \ref{[]}, for each $X\in U$, there exist operators $A_X<X<B_X$, $[A_X, B_X]\subset U$, $\phi(A_X)<\phi(X)<\phi(B_X)$, $[\phi(A_X), \phi(B_X)]\subset \phi(U)$ such that $\phi([A_X, B_X])=[\phi(A_X), \phi(B_X)]$ and  $\phi : [A_X, B_X]\to [\phi(A_X), \phi(B_X)]$ is an order isomorphism. 
By Lemma \ref{ABCD}, $\phi$ restricted to $[A_X, B_X]$ extends to a biholomorphic map $\phi_X$ defined on some open subset of $B(H)$  containing $[A_X, B_X]\cup \Pi(H)$ which maps $[A_X, B_X]$ bijectively onto $[\phi(A_X), \phi(B_X)]$ and $\Pi(H)$ biholomorphically onto itself. 
Let $X, Y\in U$. 
Suppose that there exist $A, B\in U$ with $A<B$ and $[A, B]\subset [A_X, B_X]\cap [A_Y, B_Y]$. 
By Proposition \ref{identity}, we have $\phi_X(Z)=\phi_Y(Z)$, $Z\in \Pi(H)$. 
Since $U$ is an operator domain, it is path-connected. 
Therefore, a compactness argument implies that $\phi_X(Z)=\phi_Y(Z)$, $Z\in \Pi(H)$, holds for any $X, Y\in U$. 
Hence we obtain the desired conclusion.
\end{proof}

\section{Explicit formulae}
In what follows, we fix an orthonormal basis $\{ e_\alpha \, : \, \alpha \in J \}$ of $H$.  For $X\in B(H)$, let $X^t$ denote the transpose of $X$ with respect to this orthonormal basis. 
More precisely, $X^t$ is the unique bounded linear operator acting on $H$ such that
$$
\langle X^t e_\alpha , e_\beta \rangle = \langle X e_\beta , e_\alpha \rangle, \ \ \ \alpha, \beta \in J.
$$

We first give a concrete formula for biholomorphic automorphisms of $\Pi(H)$, which might be known among specialists. 
We give its proof for completeness. 
\begin{proposition}\label{harris}
A map $\phi : {\Pi} (H) \to {\Pi} (H)$ is a biholomorphic automorphism if and only if  there exist a bounded linear bijection $T : H \to H$ and $A,B,C \in S(H)$ such that either
\[
\begin{split}
\phi (X) &= T((X-B)^{-1} + A)^{-1} T^\ast +C,\quad X  \in {\Pi} (H), \text{ or}\\
\phi (X) &= T((X^t-B)^{-1} + A)^{-1} T^\ast +C,\quad X  \in {\Pi} (H).
\end{split}
\]
Moreover, if $\phi : {\Pi} (H) \to {\Pi} (H)$ is a biholomorphic automorphism and there exists a sequence $(X_n)\subset \Pi(H)$ with $X_n\to 0$ and $\phi(X_n)\to 0$ as $n\to \infty$, then there exist a bounded linear bijection $T : H \to H$ and an operator $A \in S(H)$ such that either
\[
\begin{split}
\phi (X) &= T(X^{-1} + A)^{-1} T^\ast,\quad X  \in {\Pi} (H), \text{ or}\\
\phi (X) &= T((X^t)^{-1} + A)^{-1} T^\ast,\quad X  \in {\Pi} (H).
\end{split}
\]
\end{proposition}
\begin{proof}
We refer to \cite{Har}. 
Recall that the open unit ball $B(H)_0$ of $B(H)$ is biholomorphically equivalent to ${\Pi} (H)$  through the Cayley transform $\Psi: Y\mapsto i(Y+I)(I-Y)^{-1}$. 
The inverse is $X\mapsto (X-iI)(X+iI)^{-1}$.
See Theorem 12 of \cite{Har} with $\mathfrak{A}=B(H)$ and $V=I$.

It is easy to see that if $A, B, C, T$ are as in the statement then 
$$
X\mapsto T((X-B)^{-1}+A)^{-1}T^\ast +C
$$
and 
$$
X\mapsto T((X^t-B)^{-1}+A)^{-1}T^\ast +C
$$
are biholomorphic automorphisms of ${\Pi} (H)$. 
Indeed, it is the composition of five biholomorphic automorphisms $X\mapsto X-B$ (or $X\mapsto X^t-B$), $X\mapsto -X^{-1}$, $X\mapsto X-A$, $X\mapsto -X^{-1}$ and $X\mapsto TXT^\ast +C$ of $\Pi(H)$. 

Let $\phi$ be a biholomorphic automorphism of $\Pi(H)$. 
We know that the map $\psi:=\Psi^{-1}\circ \phi\circ\Psi$ is a biholomorphic automorphism of $B(H)_0$. 
By \cite[Theorems 2 and 3]{Har}, $\psi$ extends (uniquely) to a homeomorphism from the closed unit ball of $B(H)$ onto itself, and moreover, such an extension maps the unitary group of $H$ onto itself (see the equation (12) of \cite{Har} and the formula that appears two lines above this equation). 
Since $\Psi^{-1}$ (has a continuous extension that) maps $S(H)$ onto the collection 
$$
\{u \,:\, u\text{ is a unitary operator of }H \text{ such that }u-1\text{ is invertible}\},
$$
which is a dense open subset of the unitary group, we see that there exists a dense open subset $U\subset S(H)$ such that $\phi = \Psi\circ \psi\circ\Psi^{-1}$ has a continuous extension to $\Pi(H)\cup U$ that maps $U$ into $S(H)$. 
Fix an element $X_0\in U$. 
Define a biholomorphic map $\tilde{\phi}$ of $\Pi(H)$ by $\tilde{\phi}(X) = \phi(X+X_0) -\lim_{0<c\to 0}\phi(icI+X_0)$. 
Then $\tilde{\phi}$ has a continuous extension to $\Pi(H)\cup\{0\}$ and $\lim_{0<c\to 0}\tilde{\phi}(icI)=0$. 
Therefore, in order to complete the proof, it suffices to show the latter half of the statement of the proposition. 

Let $\phi$ be a biholomorphic automorphism of $\Pi(H)$, and suppose that there exists a sequence $(X_n)\subset \Pi(H)$ with $X_n\to 0$ and $\phi(X_n)\to 0$. 
Take operators $A_1, A_2\in S(H)$ with $\phi(iI)=A_1+iA_2$. Then $A_2>0$. 
Consider the map $f$ defined by $f(X) = T(X^{-1}+A)^{-1}T^\ast$, $X\in \Pi(H)$, where 
$$
A=A_2^{-1/2}A_1A_2^{-1/2},\quad T= A_2^{1/2}(A^2+I)^{1/2}.
$$
A straightforward calculation shows that $f(iI)=\phi(iI)$ and $f^{-1}\circ \phi(X_n)\to 0$ as $n\to \infty$. 
Consider biholomorphic automorphisms $\psi=\Psi^{-1}\circ \phi\circ\Psi$ and $g=\Psi^{-1}\circ f\circ\Psi$ of $B(H)_0$. 
Since $f(iI)=\phi(iI)$, we obtain $g^{-1}\circ \psi (0)= 0$. 
By \cite[Theorem 1]{Har} and \cite[Theorem A.9]{Mo2}, there exists a pair of unitaries $u, v$ with either $g^{-1}\circ \psi (Y)= uYv$, $Y\in B(H)_0$, or $g^{-1}\circ \psi (Y)= uY^tv$, $Y\in B(H)_0$.
Moreover, we see that $\Psi^{-1}(X_n)\to -I$ and 
$$
g^{-1}\circ \psi (\Psi^{-1}(X_n))= \Psi^{-1}\circ f^{-1}\circ \phi (X_n) \to -I
$$ 
as $n\to \infty$.
Hence we obtain $-I = -uv$, or equivalently, $v=u^\ast$. 
Thus we have either $\psi(Y) = g(uYu^\ast)$, $Y\in B(H)_0$, or $\psi(Y) = g(uY^tu^\ast)$, $Y\in B(H)_0$. 
Since $\Psi^{-1}$ is given by functional calculus, we obtain 
\[
\begin{split}
\phi(X) &=\Psi\circ \psi(\Psi^{-1}(X))\\
&= \Psi\circ g(u\Psi^{-1}(X)u^*)\\
&= \Psi\circ g\circ \Psi^{-1}(uXu^*)\\
&= f(uXu^\ast) \\
&= T((uXu^\ast)^{-1}+A)^{-1}T^\ast \\
&= Tu(X^{-1}+u^\ast Au)^{-1}(Tu)^\ast,\quad X\in \Pi(H),
\end{split}
\]
in the former case, and similarly 
$$
\phi(X)= Tu((X^t)^{-1}+u^\ast Au)^{-1}(Tu)^\ast,\quad X\in \Pi(H),
$$
in the latter case. 
This completes the proof.
\end{proof}

Let $U\subset S(H)$ be an operator domain and $\phi : U \to S(H)$ satisfy the second condition of the main theorem. 
Hence $\phi : U \cup \Pi (H) \to B(H)$ is continuous and restricts to a biholomorphic automorphism of $\Pi(H)$. 
We will show that then the third condition of the main result is satisfied, which completes the proof of Theorem \ref{punodosta}.
Fix an operator $X_0\in U$. 
Considering the map $X\mapsto \phi(X+X_0)-\phi(X_0)$ instead of $\phi$, we may assume $0 \in U$ and $\phi(0)=0$. 

By Proposition \ref{harris}, we have either
\[
\begin{split}
\phi(X) &= T(X^{-1} + A)^{-1}T^\ast,\quad X\in {\Pi} (H), \text{ or}\\
\phi(X) &= T((X^t)^{-1} + A)^{-1}T^\ast,\quad X\in {\Pi} (H),
\end{split}
\]
where $A\in S(H)$ and $T : H\to H$ is a bounded linear bijection.
Recall that the maps $X\mapsto X^t$ and $X\mapsto T^{-1}X(T^\ast)^{-1}$ are biholomorphic automorphisms of $B(H)$ that map $\Pi(H)$ onto itself, $\Pi(H)^*$ onto itself, and restrict to order automorphisms of $S(H)$. 
Composing the map $\phi$ with $X\mapsto X^t$, if necessary, and $X\mapsto T^{-1}X (T^\ast)^{-1}$, we may assume with no loss of generality that $\phi$ is of the form 
\[
\begin{split}
\phi(X) &= (X^{-1} + A)^{-1}\\
&=(X^{-1}(I+XA))^{-1}\\
&= (XA+I)^{-1}X,\quad X\in {\Pi} (H).
\end{split}
\]

We define 
\[
W_A := \{ X\in B(H)\, : \, XA+I {\rm \ \, is\ \, invertible}\} = \{ X\in B(H)\, : \, AX+I {\rm \ \, is\ \, invertible}\},
\] 
which is an open subset of $B(H)$ containing both the generalized upper half-plane and the generalized lower half-plane.
We further define a map $\Theta_A : W_A\to B(H)$ by $\Theta_A(X) = (XA+I)^{-1}X$. 
Note that 
\[
\begin{split}
\Theta_A(X)&= (XA+I)^{-1}X \\
&= (XA+I)^{-1}X(AX+I)(AX+I)^{-1}\\
&= (XA+I)^{-1}(XA+I)X(AX+I)^{-1}\\
&=X(AX+I)^{-1}. 
\end{split}
\]

\begin{lemma}\label{diverge}
Suppose that a sequence $(X_n)\subset W_A$ converges to an operator $X\in B(H)\setminus W_A$. 
Then we have $\|\Theta_A(X_n)\|\to \infty$.
\end{lemma}
\begin{proof}
The assumption implies $\| (X_nA+I)^{-1}\|\to \infty$ as $n\to\infty$. 
Hence
\[
\begin{split}
\|\Theta_A(X_n)\|\|A\|&=\|(X_nA+I)^{-1}X_n\|\|A\|\\
&\ge \|(X_nA+I)^{-1}X_nA\| = \|I- (X_nA+I)^{-1}\|\to\infty,
\end{split}
\]
which completes the proof.
\end{proof}

Because $\phi$ is continuous on $U \cup \Pi (H)$ and $\Pi (H) \subset W_A$, the above lemma yields that $U \subset W_A$.

\begin{lemma}\label{A-A}
The map $\Theta_A$ is a biholomorphic map from $W_A$ onto $W_{-A}$, and $(\Theta_A)^{-1}= \Theta_{-A}$. 
\end{lemma}
\begin{proof}
Let $X\in W_A$. 
Then 
$$
-A\Theta_A(X) +I = -AX(AX+I)^{-1} +(AX+I)(AX+I)^{-1} = (AX+I)^{-1}.
$$ 
It follows $\Theta_A(X)\in W_{-A}$ and 
$$
\Theta_{-A}\circ \Theta_{A} (X) = \Theta_{A}(X) (-A\Theta_A(X) +I)^{-1} = X(AX+I)^{-1}(AX+I) = X.
$$ 
Thus we have $\Theta_{-A}\circ \Theta_A = \operatorname{id}_{W_A}$. 
Similarly, we can prove $\Theta_A\circ \Theta_{-A} = \operatorname{id}_{W_{-A}}$. 
Hence we obtain $(\Theta_A)^{-1}= \Theta_{-A}$.
In order to see that $\Theta_A$ is biholomorphic, it suffices to show that $\Theta_A$ is holomorphic. 
This is easily seen by the fact that the maps $X\mapsto X$ and $X\mapsto (XA+I)^{-1}$ are holomorphic on $W_A$.
\end{proof}

The proof of the main theorem is almost complete. It remains only to take $\cU$ to be the connected component of $W_A$ that contains the zero operator and similarly, to denote by $\cV$ the connected component of $W_{-A}$ that contains $0$ (we do not know whether there exists $A \in S(H)$ such that $W_A$ is not connected but this does not matter here). Since clearly $\phi (\cU ) = \cV$, we are done.

Moreover, the main theorem combined with the above discussion shows that by considering the map $\Theta_A$ and its restriction $\hat{\Phi}_A$ to $\hat{U}_A := S(H)\cap W_A$ we can get the  full understanding of the structure of local order isomorphisms.
We have $\hat{\Phi}_A(X) = (XA+I)^{-1}X=X(AX+I)^{-1}=(\hat{\Phi}_A(X))^\ast$, $X\in\hat{U}_A$. 
Note that $\hat{U}_A$ is open in $S(H)$ and  $0\in \hat{U}_A$. 
Take the connected component $U_A\ni 0$ of $\hat{U}_A$ in $S(H)$. 
We also consider the restriction $\Phi_A$ of $\hat{\Phi}_A$ to $U_A$. 

It is easy to see:
\begin{itemize}
\item The map $\hat{\Phi}_A$ is a bijection from $\hat{U}_A$ onto $\hat{U}_{-A}$, and $(\hat{\Phi}_A)^{-1}= \hat{\Phi}_{-A}$, 
\item the map $\Phi_A$ is a bijection from $U_A$ onto $U_{-A}$, and $(\Phi_A)^{-1}= \Phi_{-A}$, and  
\item $U\subset U_A$. 
\end{itemize}

Let us summarize what we have obtained so far. 
\begin{theorem}\label{TPhi}
Let $U\subset S(H)$ be an operator domain with $0\in U$. 
Suppose that $\phi: U\to S(H)$ is a local order isomorphism with $\phi(0)=0$. 
Then there exist $A\in S(H)$ and an invertible operator $T\in B(H)$ such that either 
\begin{itemize}
\item for any $X\in U$, we have $X\in U_A$ and 
$$
\phi(X) = T\Phi_A(X)T^\ast =  T(XA+I)^{-1}XT^\ast, \text{ or} 
$$
\item for any $X\in U$, we have $X^t\in U_A$ and 
$$
\phi(X) = T\Phi_A(X^t)T^\ast =  T(X^tA+I)^{-1}X^tT^\ast. 
$$
\end{itemize}
Moreover, $\Phi_A : U_A\to U_{-A}$ extends to a biholomorphic map $\Theta_A: W_A\to W_{-A}$ that maps $\Pi(H)$ onto itself.
\end{theorem}

We continue with a uniqueness result. 
\begin{proposition}\label{AA'}
Let $A, A'\in S(H)$ and $T, T'\in B(H)$. 
Suppose that $T$ and $T'$ are invertible. 
\begin{enumerate}
\item If 
$$
T(XA+I)^{-1}XT^\ast= T'(XA'+I)^{-1}X(T')^\ast
$$
holds for every $X$ in some nonempty open subset of $U_A\cap U_{A'}$, then $A=A'$ and there exists a complex number $\lambda\in \C$ with $|\lambda|=1$ such that $T = \lambda T'$.
\item The equation 
$$
T(XA+I)^{-1}XT^\ast= T'(X^tA'+I)^{-1}X^t(T')^\ast
$$
cannot hold for every $X$ in any nonempty open subset of $U_A\cap \{Y^t\, :\, Y\in U_{A'}\}$. 
\end{enumerate}
\end{proposition}
\begin{proof}
Assume that the equation holds for some nonempty open set.
By Proposition \ref{identity} the same equation holds for any $X\in \Pi(H)$. 
After an easy calculation we obtain 
$$
A+X^{-1}= ((T')^{-1}T)^\ast(A'+X^{-1})(T')^{-1}T,\quad X\in \Pi(H)
$$
in (1) and 
$$
A+X^{-1}= ((T')^{-1}T)^\ast(A'+(X^t)^{-1})(T')^{-1}T,\quad X\in \Pi(H)
$$
in (2). 
The rest of the proof is easy and is left to the reader.
\end{proof} 

We may apply Theorem \ref{TPhi} to obtain previously known results. 
As an example, we give the structure of order automorphisms of $S(H)$ first given by Moln\'ar. 
Further application of our results can be found in Section \ref{apply}.
\begin{theorem}[{\cite{Mo1}}]\label{molnar}
Let $\phi: S(H)\to S(H)$ be an order automorphism. 
Then there exist $B\in S(H)$ and a bounded bijective linear operator $T: H\to H$ such that either
\[
\begin{split}
\phi(X) &= TXT^\ast +B, \ \ \ X\in S(H), \text{ or}\\
\phi(X) &= TX^tT^\ast +B, \ \ \ X\in S(H).
\end{split}
\]
\end{theorem}
\begin{proof}
It suffices to consider the case $\phi(0)=0$. 
It is clear that $\phi$ is a local order isomorphism. 
By Theorem \ref{TPhi}, there exist $A\in S(H)$ and a bounded bijective linear operator $T: H\to H$ such that $S(H)\subset U_A$ and either 
\[
\begin{split}
\phi(X) &= T\Phi_A(X)T^\ast =  T(XA+I)^{-1}XT^\ast, \ \ \ X\in S(H), \text{ or}\\ 
\phi(X) &= T\Phi_A(X^t)T^\ast =  T(X^t A+I)^{-1}X^t T^\ast, \ \ \ X\in S(H).
\end{split}
\]
Clearly, the condition $S(H)\subset U_A$ implies $A=0$, hence we obtain the desired conclusion.
\end{proof}

Recall that $X^t$ denotes the transpose with respect to some orthonormal basis $\{ e_\alpha \, : \, \alpha \in J \}$ fixed in advance. Denote by $K : H \to H$ the conjugate-linear bijection given by $K \left( \sum_{\alpha \in J} \lambda_\alpha e_\alpha \right)
=  \sum_{\alpha \in J} \overline{\lambda_\alpha} e_\alpha$. Then it is easy to see that $X^t = KX^\ast K$ for every $X \in B(H)$. Hence, if $T\in B(H)$ is invertible and we set $S = TK$, then $S$ is a conjugate-linear bijective bounded operator on $H$ and $TX^t T^\ast = SX^\ast S^\ast$, $X \in B(H)$. The map $X \mapsto SX^\ast S^\ast$ is a biholomorphic automorphism of $B(H)$, but neither of the maps $X \mapsto X^\ast$ and $X \mapsto SXS^\ast$ is holomorphic (because they are conjugate-linear). As long as we were working with biholomorphic maps it was more convenient to write $X \mapsto TX^t T^\ast$ rather than $X\mapsto SX^\ast S^\ast$, since in the first form we have the composition of two biholomorphic automorphisms $X \mapsto X^t$ and $X \mapsto TXT^\ast$. Once we restrict our attention to operator domains in $S(H)$ and we are interested in (local) order isomorphisms, the second form is more appropriate because then $X = X^\ast$. In particular, the conclusion of the last theorem may be reformulated as follows: there  exist $B\in S(H)$ and a bounded bijective linear or conjugate-linear operator $T: H\to H$ such that $
\phi(X) = TXT^\ast +B$ for every $X \in S(H)$. The corresponding reformulation of Theorem \ref{TPhi} is left to the reader.

\section{Properties of local order isomorphisms}\label{pemi}
The aim of this section is to obtain a further understanding of the structure of local order isomorphisms on operator domains. 

First, we define maximality and an equivalence relation in the collection of all local order isomorphisms.
Let $U\subset S(H)$ be an operator domain and $\phi : U\to S(H)$ a local order isomorphism. 
We say $\phi$ is maximal if there does not exist an extension of $\phi$ to a local order isomorphism on an operator domain $V\subset S(H)$ with $U\subset V$ and $U\neq V$.
Let $U'\subset S(H)$ be another operator domain and $\phi : U\to S(H)$, $\phi' : U'\to S(H)$ two local order isomorphisms. 
We say $\phi$ and $\phi'$ are equivalent if there exist order automorphisms $\psi_1, \psi_2$ of $S(H)$ such that $\phi'= \psi_2\circ \phi\circ \psi_1$. 
Having in mind the general form of order automorphisms of $S(H)$ (Theorem \ref{molnar}), it is clear that maximality is stable by equivalence. 

In the first half of the current section, we examine the equivalence relation of local order isomorphisms. 
By Lemma \ref{diverge}, the local order isomorphism $\Phi_A : U_A\to S(H)$ is maximal.
An easy application of results in the preceding sections gives

\begin{proposition}\label{rdrit}
Let $U\subset S(H)$ be an operator domain and $\phi : U\to S(H)$ a local order isomorphism. 
\begin{enumerate}
\item There exist $A\in S(H)$ and an operator domain $V\subset U_A$ such that the restriction of $\Phi_A$ to $V$ is equivalent to $\phi$. 
\item The map $\phi : U\to S(H)$ uniquely extends to a maximal local order isomorphism. 
\end{enumerate}
\end{proposition}

In particular, a maximal local order isomorphism is equivalent to $\Phi_A : U_A\to S(H)$ for some $A\in S(H)$. 

\begin{proposition}\label{TAT*}
Let $A\in S(H)$. 
Let $T: H\to H$ be a bounded linear or conjugate-linear bijection. 
Then, $\Phi_A$ is equivalent to $\Phi_{TAT^\ast}$.  
More precisely, we have the following: 
Let $\psi$ be an order automorphism of $S(H)$ given by $\psi(X) = T^*XT$, $X\in S(H)$.
Then $\Phi_{TAT^\ast}= \psi^{-1}\circ \Phi_A\circ \psi$.
\end{proposition}
\begin{proof}
Put $B=TAT^\ast$.
Let $X\in S(H)$. 
We have $X\in \hat{U}_{B}$ if and only if $XTAT^\ast+I$ is invertible. 
On the other hand, we have $\psi(X)\in \hat{U}_A$ if and only if $T^\ast XTA +I$ is invertible. 
Hence we obtain $X\in U_{B}$ if and only if $\psi(X)\in U_A$. 
Let $X\in U_{B}$. 
Then 
\[
\Phi_A\circ \psi (X) = T^*XT (AT^*XT +I)^{-1},
\]
and hence
\[
\begin{split}
\psi^{-1}\circ \Phi_A\circ \psi (X) &=(T^*)^{-1}T^*XT (AT^*XT +I)^{-1}T^{-1}\\
&=X(TAT^*X+I)^{-1} = \Phi_{B}(X). 
\end{split}
\]
\end{proof}

Proposition \ref{TAT*} gives examples of mutually equivalent local order isomorphisms. 
In Proposition \ref{UAUB} below, we will show that every pair of mutually equivalent local order isomorphisms arises in that way.

\begin{lemma}\label{AXA}
Let $A\in S(H)$ and $X\in U_A$. 
Then there exists a bounded linear bijection $T\in B(H)$ such that $\hat{\Phi}_X(A) = TAT^*$. 
\end{lemma}
\begin{proof}
Since $0, X\in U_A$ and $U_A$ is path-connected, there exists a continuous path $\tau: [0, 1]\to U_A$ such that $\tau(0)=0$ and $\tau(1)=X$. 
We may take $0= t_0<t_1<\ldots<t_k=1$ such that $\lVert (\tau(t_{j-1})A+I)(\tau(t_j)A+I)^{-1} -I\rVert$ is sufficiently small for $j=1, \ldots, k$. 
It suffices to show that there exists a bounded linear bijection $T_j: H\to H$ such that $A(\tau(t_j)A+I)^{-1} = T_jA(\tau(t_{j-1})A+I)^{-1}T_j^*$, $j\in \{1, \ldots, k\}$. 
For this, consider the biholomorphic function $f: \{z\in \C \,:\, \operatorname{Re}z>0\}\to \C\setminus (-\infty, 0]$ defined by $f(z)=z^2$. 
Then the equations $f^{-1}(z) = \overline{f^{-1}(\bar{z})}$ and 
$$
A(\tau(t_{j-1})A+I)^{-1} \left((\tau(t_{j-1})A+I)(\tau(t_j)A+I)^{-1}\right) $$ $$=
\left((\tau(t_{j-1})A+I)(\tau(t_j)A+I)^{-1}\right)^* A(\tau(t_{j-1})A+I)^{-1}
$$
imply that 
\[
\begin{split}
&A(\tau(t_{j})A+I)^{-1}\\
&= A(\tau(t_{j-1})A+I)^{-1} \left(f^{-1}\left((\tau(t_{j-1})A+I)(\tau(t_j)A+I)^{-1}\right)\right)^{2}\\
&= S^* A(\tau(t_{j-1})A+I)^{-1} S,
\end{split}
\]
where $S =  f^{-1}\left( (\tau(t_{j-1})A+I)(\tau(t_j)A+I)^{-1}\right)$.
\end{proof}

\begin{proposition}\label{formula}
Let $X_0, A\in S(H)$ satisfy $X_0\in \hat{U}_A$ (or equivalently, $A\in \hat{U}_{X_0}$) and put $B=\hat{\Phi}_{X_0}(A)$. 
Then, for $X\in S(H)$, we have $X+X_0\in \hat{U}_A$ if and only if $X\in \hat{U}_{B}$. 
Moreover, if these equivalent conditions hold, then 
\[
\hat{\Phi}_A(X+X_0) = (X_0A+I)^{-1} \hat{\Phi}_{B}(X)(AX_0+I)^{-1} + \hat{\Phi}_A(X_0).
\] 
\end{proposition}
\begin{proof}
The condition $X+X_0\in \hat{U}_A$ means $(X+X_0)A+I$ is invertible. 
The condition $X\in \hat{U}_{B}$ means 
$$
XB+I=XA(X_0A+I)^{-1} +I = (XA+X_0A+I)(X_0A+I)^{-1}
$$ 
is invertible. 
Hence they are equivalent.  
Suppose that these equivalent conditions hold. 
Then we obtain
\[
\begin{split}
&(X_0A+I)^{-1}\hat{\Phi}_{B}(X)(AX_0+I)^{-1} + \hat{\Phi}_A(X_0)\\
&= (X_0A+I)^{-1}\left(X(BX+I)^{-1}\right)(AX_0+I)^{-1} + (X_0A+I)^{-1}X_0\\
&= (X_0A+I)^{-1}\left(X((AX_0+I)^{-1}AX+I)^{-1}\right)(AX_0+I)^{-1} + (X_0A+I)^{-1}X_0\\
&= (X_0A+I)^{-1}X(AX+AX_0+I)^{-1} + (X_0A+I)^{-1}X_0\\
&= (X_0A+I)^{-1}(X+X_0AX+X_0AX_0+X_0) (AX+AX_0+I)^{-1}\\
&= (X+X_0)(A(X+X_0)+I)^{-1} = \hat{\Phi}_A(X+X_0).
\end{split}
\]
\end{proof}

If in the above proposition we replace the assumption $X_0\in \hat{U}_A$ by a stronger one $X_0\in U_A$, then using the fact that $X \mapsto X+X_0$ is a homeomorphism of $\hat{U}_B$ onto $\hat{U}_A$
we conclude that  $X+X_0\in U_A$ if and only if $X\in U_{B}$ and in this case we have
\[
{\Phi}_A(X+X_0) = (X_0A+I)^{-1} {\Phi}_{B}(X)(AX_0+I)^{-1} + {\Phi}_A(X_0).
\]

\begin{proposition}\label{UAUB}
Let $A, B\in S(H)$. 
Suppose that $U\subset U_A$, $V\subset U_B$ are operator domains. 
Let $\phi$ (resp.\ $\psi$) denote the restriction of $\Phi_A$ to $U$ (resp.\ $\Phi_B$ to $V$). 
If $\phi$ is equivalent to $\psi$, then there exists a bounded linear or conjugate-linear bijection $T\in B(H)$ such that $B=TAT^\ast$. 
\end{proposition}
\begin{proof}
There exist order automorphisms $\psi_1, \psi_2: S(H)\to S(H)$ such that $\psi = \psi_2\circ\phi\circ\psi_1$. 
Considering extensions to maximal local order isomorphisms of both sides, we obtain $\Phi_B = \psi_2\circ\Phi_A\circ\psi_1$. 
In particular, for $X\in S(H)$, we have $X\in U_B$ if and only if $\psi_1(X)\in U_A$. 

By Theorem \ref{molnar}, there exist bounded linear or conjugate-linear bijections $T_1, T_2: H\to H$ and $C_1, C_2\in S(H)$ such that  
$$
\psi_1(X)= T_1XT_1^* +C_1,\quad\psi_2(X) = T_2XT_2^* + C_2,\quad X\in S(H).
$$
Put $D:= \hat{\Phi}_{C_1}(A) = A(C_1A+I)^{-1}$. 
Let $X \in  U_B$. By Lemma \ref{formula}, we have
\[
\begin{split}
\Phi_B(X)&= \psi_2\circ\Phi_A\circ\psi_1(X) \\
&= \psi_2\circ\Phi_A(T_1XT_1^* +C_1) \\
&= \psi_2\left((C_1A+I)^{-1}\Phi_D(T_1XT_1^*)(AC_1+I)^{-1} +\Phi_A(C_1)\right) \\
&= T_2(C_1A+I)^{-1}\Phi_D(T_1XT_1^*)(AC_1+I)^{-1}T_2^* +T_2\Phi_A(C_1)T_2^* + C_2.
\end{split}
\]
Since $\Phi_{B}(0)=0$, we obtain $T_2\Phi_A(C_1)T_2^* + C_2=0$. 
By Lemma \ref{TAT*}, we obtain 
\[
\begin{split}
\Phi_{B} (X) &= T_2(C_1A+I)^{-1}\Phi_D(T_1XT_1^*)(AC_1+I)^{-1}T_2^* \\
&= T_2(C_1A+I)^{-1}T_1\Phi_{T_1^*DT_1}(X)T_1^*(AC_1+I)^{-1}T_2^*.
\end{split}
\]
Proposition \ref{AA'} implies $B=T_1^*DT_1=T_1^*A(C_1A+I)^{-1}T_1$. 
Since $C_1=\psi_1(0)\in U_A$, Lemma \ref{AXA} implies that there exists a bounded linear or conjugate-linear bijection $T\in B(H)$ with $B=TAT^*$. 
\end{proof}

As a consequence of Propositions \ref{TAT*} and \ref{UAUB}, we obtain
\begin{corollary}\label{number}
Let $A, B\in S(H)$. 
Two maximal order isomorphisms $\Phi_A$ and $\Phi_B$ are equivalent if and only if
there exists  a bounded linear or conjugate-linear bijection $T: H\to H$ such that $B=TAT^*$. 
If $2\le n=\dim H<\infty$, then the number of equivalence classes of maximal local order isomorphisms is $(n+2)(n+1)/2$.  
\end{corollary}
\begin{proof}
The former half is clear. 
In the matrix language it can be reformulated as follows: 
 $\Phi_A$ and $\Phi_B$ are  equivalent if and only if there exists an invertible $n \times n$ complex matrix $T$ such that $A = TBT^\ast$ or $A = TB^t T^\ast$. Since hermitian matrices $B$ and $B^t$ have the same eigenvalues they are unitarily similar.
Consequently, $\Phi_A$ and $\Phi_B$ are equivalent if and only if there exists an invertible $n \times n$ complex matrix $T$ such that $A = TBT^\ast$. By the Sylvester's law of inertia this is equivalent to the condition that the number of positive eigenvalues of $A$ and $B$ coincide and the same is true for the number of negative eigenvalues (here, each eigenvalue is counted with its multiplicity). 
By Proposition \ref{rdrit}, the number of equivalence classes of maximal order isomorphisms of matrix domains in $S_n$ is $(n+2)(n+1)/2$.
\end{proof}

It is clear that an order isomorphism between two operator domains is a local order isomorphism. 
The converse does not hold. 
However, we know that the maximal local order isomorphism $X\mapsto -X^{-1}$, from the operator interval $(0, \infty)$ onto $(-\infty, 0)$, is an order isomorphism.  
In  the second half of this section, we consider the following question: 
When is a local order isomorphism an order isomorphism?
By Proposition \ref{rdrit}, it suffices to consider a restriction of the map $\Phi_A: {U}_A\to {U}_{-A}$ to an operator domain $U\subset U_A$. 

\begin{proposition}\label{interval}
Let $A\in {S}(H)$. 
Suppose that $X\in U_A$ satisfies $X\ge 0$. 
Then the condition $[0, X]\subset {{U}}_A$ is equivalent to $X^{1/2}AX^{1/2}\ge -I$. 
Moreover, if these equivalent conditions hold, then $\Phi_A$ restricts to an order isomorphism from $[0, X]$ onto $[0, \Phi_A(X)]$.
\end{proposition}
\begin{proof}
Since $X\in U_A$, we have $-1\notin\sigma(XA)\setminus\{0\} = \sigma(X^{1/2}AX^{1/2})\setminus\{0\}$.
If $X^{1/2}AX^{1/2}\ngeq -I$, then there exists a positive real number $c$ such that $0<c< 1$ and $-1\in\sigma((cX)^{1/2}A(cX)^{1/2})$. 
We know that $-1\in\sigma((cX)^{1/2}A(cX)^{1/2})$ if and only if $cXA+I$ is non-invertible, thus $cX\notin {{U}}_A$. 
If $X^{1/2}AX^{1/2}\ge -I$, then we actually have $X^{1/2}AX^{1/2}> -I$. 
It follows that $-I< TX^{1/2}AX^{1/2}T^*$ for any contractive operator $T$ on $H$,  i.e., for any $T\in B(H)$ with $\|T\|\le 1$. 
Indeed, there exists a positive real number $0<\varepsilon\le 1$ such that 
$$
\langle TX^{1/2}AX^{1/2}T^* x, x\rangle = \langle X^{1/2}AX^{1/2}(T^*x), T^*x\rangle \geq -(1-\varepsilon) \lVert T^* x\rVert^2\geq -(1-\varepsilon)
$$
for any unit vector $x\in H$.
If $Y\in [0, X]$, then by Douglas' lemma there exists a contractive operator $T\in B(H)$ such that $Y^{1/2}= TX^{1/2}= X^{1/2}T^*$. 
We have $-I<TX^{1/2}AX^{1/2}T^* = Y^{1/2}AY^{1/2}$, hence $\sigma(Y^{1/2}AY^{1/2})\cap (-\infty, -1]= \emptyset$, and $AY+I$ is invertible. 

Suppose that $[0, X]\subset {{U}}_A$. 
Theorem \ref{1} implies $\Phi_A(X)\ge \Phi_A(0)=0$.
We have $-\Phi_A(X)A = f(XA)$, where $f$ is a biholomorphic function defined by 
\[
f(z) = -\frac{z}{z+1}
\]
on $\C\setminus \{-1\}$. 
Since $\sigma(XA)\subset\sigma(X^{1/2}AX^{1/2})\cup \{0\}\subset (-1, \infty)$, the spectral mapping theorem implies $\sigma(-\Phi_A(X)A) = f(\sigma(XA)) \subset (-1, \infty)$.
Thus we obtain $[0, \Phi_A(X)]\subset U_{-A}$. 
Since $(\Phi_A)^{-1}=\Phi_{-A}$, Theorem \ref{1} implies that $\Phi_A$ restricts to an order isomorphism from $[0, X]$ onto $[0, \Phi_A(X)]$.
\end{proof}

\begin{proposition}\label{restriction}
Let $U\subset U_A$ be an operator domain. 
Then the map $\Phi_A$ restricts to an order preserving map on $U$ if and only if the following condition holds: 
For any pair $X, Y\in U$ with $X\le Y$, $[X, Y]\subset U_A$ holds. 
\end{proposition}
\begin{proof}
Theorem \ref{1} gives one direction. 
Suppose that there exist a pair $X, Y\in U$ with $X\le Y$, $[X, Y]\not\subset U_A$.
Take a positive real number $\varepsilon>0$ such that $Y+\varepsilon I\in U$. 
Then trivially we have $[X, Y+\varepsilon I]\not\subset U_A$. 
By Proposition \ref{formula}, it suffices to show: 
If $X\in U_A$, $X>0$ and $[0, X]\not\subset U_A$, then $\Phi_A(X)\ngeq 0$. 
Suppose $X\in U_A$ satisfies $X>0$ and $[0, X]\not\subset U_A$.
By Proposition \ref{interval}, the condition $[0, X]\not\subset U_A$ yields $X^{1/2}AX^{1/2}\ngeq-I$. 
We obtain $X^{-1}+A \ngeq0$, which implies $\Phi_A(X)=(X^{-1}+A)^{-1}\ngeq0$.
\end{proof}

For an operator $A\in {S}(H)$, let $s(A)$ denote the support projection of $A$, namely, $s(A) = \chi_{(-\infty, \infty)\setminus\{0\}}(A)$. 
We also use the notation $A_+ := A\chi_{(0, \infty)}(A)$ and $A_- := -A\chi_{(-\infty, 0)}(A)$, which are positive and negative parts of $A$, respectively. 

The rest of this section aims to prove
\begin{theorem}\label{compact}
Let $A\in {S}(H)$. 
Then the following two conditions are equivalent: 
\begin{enumerate}
\item At least one of the operators $A_+$ and $A_-$ is compact. 
\item The map $\Phi_A$ is an order isomorphism from ${{U}}_A$ onto ${{U}}_{-A}$.
\end{enumerate}
\end{theorem}

We first consider the case either $A_+$ or $A_-$ has finite rank. 

\begin{lemma}\label{minmax}
Let $m\in \{0, 1, \ldots\}\cup \{\infty\}$ satisfy $m\leq \dim H$. 
Let ${V}$ denote the collection of all invertible $A\in {S}(H)$ with $\rank A_+=m$. 
\begin{enumerate}
\item The set ${V}$ is open and closed in the collection of all invertible operators in ${S}(H)$.  
\item  Assume $m<\infty$. Let $A_0, A_1\in {V}$ satisfy $A_0\leq A_1$. 
If $A\in [A_0, A_1]$, then $A\in V$. 
\end{enumerate}
\end{lemma}
\begin{proof}
(1) It suffices to prove the openness of $V$ for each $m$. 
Let $A\in V$ for a fixed $m$.
Since the collection of all invertible operators is open, we may take a positive real number $\varepsilon$ such that any element in $[A-\varepsilon I, A+\varepsilon I]$ is invertible. 
Clearly, we have $A-\varepsilon I, A+\varepsilon I\in V$. 
Let $B\in [A-\varepsilon I, A+\varepsilon I]$.
Since $A-\varepsilon I\le B$, the min-max theorem implies that 
$m=\rank (A-\varepsilon I)_+\leq \rank B_+$. 
Since $B\le A+\varepsilon I$, we also have $\rank B_+ \leq \rank (A+\varepsilon I)_+ = m$. 
Thus we obtain $\rank B_+ = m$.

(2) Let $A\in [A_0, A_1]$. 
The min-max theorem implies 
$\rank A_+\geq \rank (A_0)_+=m$. 
We may take a positive real number $\varepsilon>0$ such that $\rank (A_1+\varepsilon I)_+ = m$. 
The min-max theorem implies $\rank (A+\varepsilon I)_+ \leq \rank (A_1+\varepsilon I)_+ = m$. 
Since $m<\infty$, we see $\rank A_+= m$ and $A$ is invertible.
\end{proof}

For a projection $P\in B(H)$, let us use the symbol $P^{\perp}:= I-P$.

\begin{lemma}\label{finite}
Let $A\in S(H)$. 
If either $A_+$ or $A_-$ has finite rank, then $\Phi_A$ is an order isomorphism from ${U}_A$ onto ${U}_{-A}$.
\end{lemma}
\begin{proof}
Put $S:=s(A_+)-s(A_+)^{\perp}$, which is a self-adjoint unitary.
Let $X\in S(H)$. 
We have 
$$
\sigma(XA)\setminus \{0\}= \sigma(X\lvert A\rvert^{1/2} S\lvert A\rvert^{1/2})\setminus \{0\}
= \sigma(\lvert A\rvert^{1/2}X\lvert A\rvert^{1/2} S)\setminus \{0\}.
$$
Hence we have $X\in \hat{U}_A$ if and only if $\lvert A\rvert^{1/2}X\lvert A\rvert^{1/2} S+I$ is invertible, which is in turn equivalent to the condition $\lvert A\rvert^{1/2}X\lvert A\rvert^{1/2} + S$ is invertible.

Suppose that $A_+$ has rank $m<\infty$. 
We prove that $\Phi_A$ is order preserving. 
The preceding paragraph shows that $\lvert A\rvert^{1/2}X\lvert A\rvert^{1/2} + S$ is invertible for any $X\in {U}_A$.
Moreover, since $0\in U_A$ is connected, Lemma \ref{minmax} implies that the positive part
$\left(\lvert A\rvert^{1/2}X\lvert A\rvert^{1/2} + S\right)_+$
has rank $m$ for any $X\in {U}_A$. 
Suppose that $X, Y\in {U}_A$ satisfy $X\leq Y$. 
Let $Z\in [X, Y]$. 
Then we have the order relation 
$$
\lvert A\rvert^{1/2}X\lvert A\rvert^{1/2} + S\leq \lvert A\rvert^{1/2}Z\lvert A\rvert^{1/2} + S\leq \lvert A\rvert^{1/2}Y\lvert A\rvert^{1/2} + S.
$$
By Lemma \ref{minmax}, $\lvert A\rvert^{1/2}Z\lvert A\rvert^{1/2} + S$ is invertible. 
Hence $ZA+I$ is invertible for any $Z\in [X, Y]$. 
It follows that $[X, Y]\subset {U}_A$, and by Proposition \ref{restriction}, $\Phi_A(X)\leq \Phi_A(Y)$ holds. 

The case $A_-$ has finite rank can be considered in a similar way. 
The same argument shows that $\Phi_{A}^{-1} = \Phi_{-A}$ is also order preserving.
\end{proof}

\begin{proof}[Proof of Theorem \ref{compact}]
$(1)\Rightarrow(2)$
Suppose that $A_+$ is compact. 
We show that $\Phi_A$ is order preserving. 
By Lemma \ref{AXA} and Propositions \ref{formula}, \ref{restriction}, all we need to do is to prove that if $X \in U_A$ and $X \geq 0$ then $[0,X] \in U_A$. 

By $X \in U_A$, there exists a path $\tau : [0,1] \to U_A$ such that $\tau (0) = 0$ and $\tau (1) = X$. 
Since $A_+$ is compact we can find a sequence $(A_n) \subset S(H)$ such that $\lim A_n = A$ and $(A_n)_+$ is of finite rank for every positive integer $n$.
For every $t \in [0,1]$ the operator $\tau (t) A + I$ is invertible. It follows that there exists $\delta_t > 0$ such that if
$$
s \in (t-\delta_t , t + \delta_t) \cap [0,1] \ \ \ {\rm and} \ \ \ B \in S(H) \ \ \  {\rm and} \ \ \    \| B - A\| <\delta_t,
$$
then $\tau(s) B +I$ is invertible. Using the compactness we see that there exists $\delta > 0$ such that $\tau (s) B + I$ is invertible for every $s \in [0,1]$ and every $B \in S(H)$ with $\| B - A \| < \delta$.
Thus, after removing finitely many elements of the sequence we may assume that $\tau (s) A_n +I$ is invertible for every $s \in [0,1]$ and every positive integer $n$.
It follows that $X \in U_{A_n}$. 
By Lemma \ref{finite} and Proposition \ref{restriction}, we conclude that $[0,X] \subset U_{A_n}$. 
By Proposition \ref{interval} we have $X^{1/2} A_n X^{1/2} \ge -I$ for each $n$.
Taking the limit $n \to \infty$, we obtain $X^{1/2} A X^{1/2} \ge -I$.
By Proposition \ref{interval} again, we have $[0,X] \subset U_A$, as desired.

The case $s(A_-)$ is compact can be considered in a similar way, or we may use the equation $\Phi_{-A}(X) = (-XA+I)^{-1}X = -\Phi_A(-X)$ for $X\in U_{-A}$ as well. 
The same argument shows that $\Phi_{A}^{-1} = \Phi_{-A}$ is also order preserving.

$(2)\Rightarrow(1)$
Suppose that both $A_+$ and $A_-$ are noncompact. 
Then we may take a positive real number $\varepsilon >0$ such that both $\chi_{[\varepsilon, \infty)}(A)$ and $\chi_{(-\infty, -\varepsilon]}(A)$ have infinite rank. 
We can find a projection $P\in S(H)$ with $\chi_{[\varepsilon, \infty)}(A)\leq P\leq \chi_{(-\infty, -\varepsilon]\cup[\varepsilon, \infty)}(A)$ such that $P$ and $\chi_{[\varepsilon, \infty)}(A)$ are unitarily equivalent in $K:=\chi_{(-\infty, -\varepsilon]\cup[\varepsilon, \infty)}(A)H$, $P\neq \chi_{[\varepsilon, \infty)}(A)$ and $PA=AP$. 
Take a unitary $u=e^{iB}\in B(K)$, $B\in S(K)$, such that $u\chi_{[\varepsilon, \infty)}(A)u^* = P$. 
Put $S:= \chi_{[\varepsilon, \infty)}(A)- \chi_{(-\infty, -\varepsilon]}(A)$, which can be considered a self-adjoint unitary in $B(K)$. 
We define a path $\tau: [0, 1]\to S(K)$ by 
$$
\tau(t) = \lvert A\rvert^{-1/2}(e^{itB} Se^{-itB} -S) \lvert A\rvert^{-1/2}
$$
(which is well-defined since $A$ restricted to $K$ is invertible).
We identify each element $X\in S(K)$ with $\chi_{(-\infty, -\varepsilon]\cup[\varepsilon, \infty)}(A)X\chi_{(-\infty, -\varepsilon]\cup[\varepsilon, \infty)}(A)\in S(H)$. Then 
$$
\lvert A \rvert^{1/2}\tau(t)\lvert A\rvert^{1/2} + (s(A_+)-s(A_+)^{\perp})= e^{itB} Se^{-itB} +(\chi_{(0, \varepsilon)}(A)- \chi_{(-\varepsilon, 0]}(A))
$$
is invertible in $B(H)$ for any $t\in [0, 1]$. 
In particular, we have $\tau(1)\in U_A$ (see the first paragraph of the proof of Lemma \ref{finite}). 
However, we have 
\[
\tau(1) = \lvert A\rvert^{-1/2} 2(P-\chi_{[\varepsilon, \infty)}(A)) \lvert A\rvert^{-1/2} = 2\lvert A\rvert^{-1}(P-\chi_{[\varepsilon, \infty)}(A))\geq 0
\]
and
\[
\begin{split}
\Phi_A(\tau(1)) &= \Phi_A(2\lvert A\rvert^{-1}(P-\chi_{[\varepsilon, \infty)}(A)))\\
&= 2\lvert A\rvert^{-1}(P-\chi_{[\varepsilon, \infty)}(A))\cdot (A\cdot 2\lvert A\rvert^{-1}(P-\chi_{[\varepsilon, \infty)}(A))+I)^{-1}\\
&= -2\lvert A\rvert^{-1}(P-\chi_{[\varepsilon, \infty)}(A))\ngeq 0=\Phi_A(0).
\end{split}
\]
\end{proof}

\section{Finite-dimensional case}
When dealing with the finite-dimensional case we will always identify operators with matrices.
We denote by $E_n$ the effect algebra
on the $n$-dimensional Hilbert space, $E_n = [0,I]$, that is, $E_n$ is the set of all $n \times n$ hermitian matrices whose all eigenvalues belong to the unit interval $[0,1]$. As already indicated in Introduction we will begin with  an automatic continuity result for order embeddings of matrix domains. First we need some technical lemmas. 

\begin{lemma}\label{more,more}
Let $A,B \in S_n$ be two different matrices with $A \le B$. Denote $r = {\rm rank}\, (B-A)$. Then there exists a continuous order isomorphism 
from $E_r$ onto $[A,B]$.
\end{lemma}

Actually, we will prove more. We will find an order isomorphism that is not just continuous but a restriction of an affine map.

\begin{proof}
After applying a translation $X \mapsto X-A$ we may assume with no loss of generality that $A=0$ and $B \in S_n$ satisfies $B \ge 0$ and
${\rm rank }\, B =r$. Further, there is no loss of generality in assuming that the matrix $B$ is diagonal with the first $r$ diagonal entries positive and all others equal to zero. Applying a congruence transformation $X \mapsto TXT^\ast$, where $T$ is a suitable diagonal invertible element in $M_n$, we can  finally assume that $B$ is
the diagonal projection with the first  $r$ diagonal entries equal to one and all others equal to zero. It is now clear that $[A,B]$ is
the set of all $n \times n$ matrices whose upper left $r \times r$ corner is an element of $E_r$ and all entries outside the upper left $r \times r$ corner are zero.
\end{proof}

\begin{lemma}\label{muci}
Let $m,n$ be positive integers and $\phi : E_n \to S_m$ an order embedding. Then the set of all points $c \in (0,1)$ for which $\phi$ is not continuous at $cI$ is at most countable.
\end{lemma}

\begin{proof}
Consider the order embedding $f : [0,1] \to \mathbb{R}$ defined by $f(t) := {\rm tr}\, (\phi (tI))$. Here, ${\rm tr}\, (\cdot )$ stands for the trace function. We know that the set $N$ of all points $c \in (0,1)$ at which $f$ is not continuous is at most countable. Let $c \in (0,1) \setminus N$ and $\varepsilon >0$. Then we can find $\delta$, $0 < \delta < \min \{ c, 1-c \}$, such that for every real $t \in [ c- \delta , c + \delta ]$ we have $| f(c) -  f(t) | \le \varepsilon$.
In particular,  $| {\rm tr}\, (\phi (cI) - \phi ((c \pm \delta) I)) | \le \varepsilon$. Since $\phi ((c - \delta)I) \le \phi (cI)$ and $\phi ((c + \delta)I) \ge \phi (cI)$
we conclude that
$\phi (cI) - \varepsilon I  \le \phi ((c-\delta)I) \le \phi ((c+\delta)I) \le \phi (cI) + \varepsilon I$ (note that we are using the same symbol $I$ to denote the $n \times n$ identity matrix and the $m \times m$ identity matrix). This further implies that for every $X \in [ (c-\delta)I , ( c+ \delta)I ]$ we have $\phi (cI) - \varepsilon I \le \phi (X) \le \phi (cI) + \varepsilon I$, and thus, $\phi$ is continuous at $cI$.
\end{proof}

\begin{lemma}\label{ankar}
Let $m,n$ be positive integers and assume that there exists an order embedding $\phi : E_n \to S_m$. Then $n \le m$. 
\end{lemma}

\begin{proof}
By induction on $n$. The case $n=1$ is trivial. Assume that the statement holds for $n= k-1$ ($ \ge 1$). We prove the statement for $n=k$. By Lemma \ref{muci}, we can find a real number $c$, $0 < c < 1$, such that $\phi$ is continuous at $cI$. We take any projection $P\in E_n$ of rank $k-1$. Observe that the rank of $(P+cP^\perp) - cI$ is $k-1$. By Lemma \ref{more,more} and the induction hypothesis the rank of $\phi (P + cP^\perp) - \phi (cI)$ is no smaller than $k-1$. 

We claim that $(\phi (P + cP^\perp) - \phi (cI)) + (\phi (cP + dP^\perp) - \phi (cI))$ has rank $\ge k$ for every $d$, $c < d \le 1$. If this does not hold, then there exists some $d > c$ such that the range of $\phi (cP + dP^\perp) - \phi (cI)$ is contained in the range of $\phi (P + cP^\perp) - \phi (cI)$.
Since $\phi$ is continuous at $cI$ there exists $d'$ with $c < d' < d$ such that $\phi (cP + d' P^\perp) - \phi (cI) \le \phi (P + cP^\perp) - \phi (cI)$, and therefore $\phi (cP + d' P^\perp)  \le \phi (P + cP^\perp)$. Since $\phi$ is an order embedding we necessarily have $cP + d' P^\perp  \le P + cP^\perp$, a contradiction. Thus, we have shown that $m \ge k$.
\end{proof}

\begin{corollary}\label{kadkrene}
Let $n$ be a positive integer and $A,B \in S_n$ hermitian matrices with $A < B$. Assume that $\phi : [A,B] \to S_n$ is an order embedding. Then $\phi (A) < \phi (B)$.
\end{corollary}

\begin{proof}
By Lemma \ref{more,more}, the interval $[A,B]$ is order isomorphic to $E_n$ while the interval $[\phi (A), \phi (B)]$ is order isomorphic to $E_m$, where $m$ is the rank of $\phi (B) - \phi (A)$. Since $\phi$ maps $[A,B]$ into $[\phi (A), \phi (B)]$, Lemma \ref{ankar} yields that $m=n$.
\end{proof}

The following result can be found in \cite{Se6}.

\begin{theorem}\label{staro}
{\rm (\cite[Theorem 2.4]{Se6}).} Let $n \ge 2$. For every pair $A,B \in (0,I) \subset E_n$ there exists an order automorphism $\phi : E_n \to E_n$ such that $\phi (A) = B$.
\end{theorem}

By Theorem \ref{fpfq} every order automorphism of $E_n$ is continuous. Using Lemma \ref{more,more} and the remark that follows that lemma we can easily deduce the next statement.

\begin{corollary}\label{jebemt}
Let $n \ge 2$ and $A,B,C,D,E,F \in S_n$ with $A< B < C$ and $D < E < F$. Then there exists a continuous order isomorphism $\phi : [A,C] \to [D,F]$ such that $\phi (B) = E$.
\end{corollary}

Let $A\in S_n$ and assume that $A \ge 0$. Clearly, $A_0$, the restriction of $A$ to the range of $A$, is an invertible operator on the range of $A$. We denote by $A^\dagger$ the Moore-Penrose inverse of $A$, that is, the restriction of $A^\dagger$ to the null space of $A$ is the zero operator, while the restriction of $A^\dagger$ to the range of $A$ is the inverse of $A_0$.

\begin{lemma}\label{naokoli}
Let $A,R \in S_n$ satisfy $A \ge 0$, $R\ge 0$, and ${\rm rank}\, R = 1$. Then $R \le A$ if and only if the range of $R$ is contained in the range of $A$ and 
${\rm tr}\, (A^\dagger R) \le 1$.
\end{lemma}

\begin{proof}
Assume first that $R \le A$. Then the range of $R$ is a subspace of the range of $A$ and we have
$$
(A^\dagger)^{1/2} R (A^\dagger)^{1/2} \le (A^\dagger)^{1/2} A (A^\dagger)^{1/2} = P,
$$
where $P$ denotes the projection onto the range of $A$. After replacing $A, A^\dagger , R,P$  with their restrictions to the range of $A$, we can rewrite the above inequality as $A^{-1/2} R A^{-1/2} \le I$ and since $A^{-1/2} R A^{-1/2}$ is a positive rank-one operator this is equivalent to ${\rm tr}\, (A^{-1} R ) =  {\rm tr}\, (A^{-1/2} R A^{-1/2} )  \le 1$. 

The proof of the other direction is now easy and is left to the reader.
\end{proof}

By $P_{n}^k \subset E_n$ we denote the set of all $n \times n$ projections of rank $k$ and by ${\cal S} \subset \mathbb{R}^3$ the euclidean sphere with centre $0$ and radius $1/2$. It is well known that $P_{2}^1$ is affine isometrically isomorphic to ${\cal S}$. An isomorphism (known as
Bloch's representation) is given by ${\cal S} \ni (x,y,z) \mapsto \left[ \begin{matrix}  x + 1/2 & y + iz \cr y - iz & -x + 1/2 \cr \end{matrix} \right]$. In addition to the usual euclidean distance we also consider the geodesic distance $d_g$ on ${\cal S}$ defined for $x,y \in {\cal S}$ as the unique 
real number satisfying $0 \le d_g (x,y) \le \pi /2$ and $| x-y | = \sin d_g (x,y)$. Here, $| \cdot |$ denotes the euclidean norm.

\begin{lemma}\label{coron}
Let $P,Q \in P_{2}^1$ be orthogonal and $0 \le c < d \le 1$ real numbers. Then a projection $E \in P_{2}^1$ satisfies $dE \le P + cQ$ if and only if $\| P - E \| \le \sqrt{ (c-cd) / (d - cd) }$. In particular, we have $(1/2) E \le P + cQ$ if and only if  $\| P - E \| \le \sqrt{ c / (1 - c) }$.
\end{lemma}

\begin{proof}
The statement is trivial when $c=0$ or $d=1$. Thus we assume that $0< c < d <1$.
Taking an appropriate orthonormal basis of $\mathbb{C}^2$, we may assume 
$P=  \begin{bmatrix}1 & 0 \\ 0 & 0 \end{bmatrix}$, $Q=  \begin{bmatrix}0 & 0 \\ 0 & 1 \end{bmatrix}$ and $E=  \begin{bmatrix}s^2 & s\sqrt{1-s^2} \\ s\sqrt{1-s^2} & 1-s^2 \end{bmatrix}$ for some real number $0\leq s\leq 1$. Note that $P-E$ is a trace zero $2 \times 2$ hermitian matrix and therefore $\| P-E\| = \sqrt{ - \det (P-E)} = \sqrt{ 1 - s^2}$, and
$$
(P+cQ)^{-1} = \begin{bmatrix}1 & 0 \\ 0 & c^{-1} \end{bmatrix}.
$$
The proof can be completed by a straightforward application of Lemma \ref{naokoli}.
\end{proof}

\begin{lemma}\label{tvvadba}
Let $\psi : {\cal S} \to {\cal S}$ be a homeomorphism. Suppose that for every $x \in {\cal S}$ there exists an increasing bijection $g_x : [0,1] \to [0, 1]$ such that $| \psi (y) - \psi (x) | = g_x (| y-x |)$ for every $y\in {\cal S}$. Then $\psi$ is a surjective isometry on ${\cal S}$ with respect to the euclidean distance.
\end{lemma}

\begin{proof}
It is easy to see that $\psi$ preserves antipodes. 
Let $x\in {\cal S}$ and $0<c<1$. 
Then $C_{x, c}:=\{y\in {\cal S}\, : \, \lvert y-x\rvert = c\}$ is mapped onto $C_{\psi(x), g_x(c)}=\{y\in {\cal S}\, : \, \lvert y-\psi(x)\rvert = g_x (c)\}$. 
(These collections are circles in ${\cal S}$.) 
Therefore, for any circle $C$ in ${\cal S}$, the image $\psi(C)$ is a circle in ${\cal S}$. 

Let $x_1, x_2$ be distinct elements in ${\cal S}$. 
Suppose that $x_1$ and $x_2$ are not antipodal in ${\cal S}$. 
Take the great circle $C_0$ in ${\cal S}$ that contains the two points $x_1, x_2$. 
Then the image $\psi(C_0)$ is a circle in ${\cal S}$.
Let $x\in C_0\setminus\{x_1, x_2, -x_1, -x_2\}$. 
An elementary calculation shows that we can take $c_1, c_2\in (0, 1)$ such that $C_{x_1, c_1}\cap C_{x_2, c_2} = \{x\}$. 
It follows that $C_{\psi(x_1), g_{x_1}(c_1)}\cap C_{\psi(x_2), g_{x_2}(c_2)} = \{\psi(x)\}$. 
This equation implies that $\psi(x)$ lies in the great circle that contains the two points $\psi(x_1), \psi(x_2)$. 
Hence the image $\psi(C_0)$ is the great circle that contains the two points $\psi(x_1), \psi(x_2)$.

We consider the geodesic distance $d_g$ for a while.
We already know that $\psi$ preserves the geodesic distance $\pi/2$. 
That is, we have $d_g(\psi(x_1), \psi(x_2))= \pi/2$ if $d_g(x_1, x_2)= \pi/2$.
Suppose that the geodesic distance $d$ with $0<d\leq\pi/2$ is preserved by $\psi$. 
We prove that the geodesic distance $d/2$ is also preserved by $\psi$. 
Let $x_1, x_2\in {\cal S}$ satisfy $d_g(x_1, x_2)=d/2$. 
Take the great circle $C_0$ in ${\cal S}$ that contains the two points $x_1, x_2$. 
Take the unique point $x_0\in C_0\setminus\{x_2\}$ such that $d_g(x_1, x_0)=d/2$. 
Then we have $d_g(x_0, x_2) = d$ and hence $d_g(\psi(x_0), \psi(x_2))=d$. 
It follows that 
\[
\{\psi(x_0), \psi(x_2)\} \subset \psi(C_{x_1, \sin(d/2)}\cap C_0) = \psi(C_{x_1, \sin(d/2)})\cap\psi(C_0). 
\] 
Note that $\psi(C_{x_1, \sin(d/2)})$ is a circle with centre $\psi(x_1)\in \psi(C_0)$. 
Note also that $\psi$ restricted to $C_0$ is a homeomorphism onto $\psi(C_0)$ and preserves antipodes. 
Hence we obtain $d_g(\psi(x_1), \psi(x_2))= d/2$. 

We next prove the following in a similar way: 
If $\psi$ preserves the geodesic distances $d_1, d_2>0$ with $d_1+d_2<\pi/2$, then $\psi$ also preserves the geodesic distance $d_1+d_2$. 
Let $x_1, x_2\in {\cal S}$ satisfy $d_g(x_1, x_2)=d_1+d_2$. 
Take the great circle $C_0$ in ${\cal S}$ that contains the two points $x_1, x_2$. 
Take the unique point $x_0\in C_0$ such that $d_g(x_1, x_0)=d_1$ and $d_g(x_0, x_2)=d_2$. 
Then we have $d_g(\psi(x_1), \psi(x_0))=d_1$ and $d_g(\psi(x_0), \psi(x_2))=d_2$. 
Since $\psi$ restricted to $C_0$ is a homeomorphism onto $\psi(C_0)$ and preserves antipodes, we obtain $d_g(\psi(x_1), \psi(x_2)) = d_1+d_2$.

It follows that $\psi$ preserves the geodesic distance $k\pi/2^l$ whenever $k, l$ are positive integers with $k/2^l\leq 1/2$.
Since $\psi$ is a homeomorphism, we see that $\psi$ is an isometry with respect to the geodesic distance, thus also an isometry with respect to the euclidean distance in $\mathbb{R}^3$. 
\end{proof}

The next result is the crucial step in proving the automatic continuity for order embeddings of matrix domains. From now till the end of this section we will always assume that $n \ge 2$.

\begin{theorem}\label{autcon1}
Let $\phi : E_n \to S_n$ be an order embedding that is continuous at $0$ and $I$. Then $\phi$ is an order isomorphism of $E_n$ onto $[ \phi (0), \phi (I) ]$.
\end{theorem}

It is easy to see that the assumption of continuity at the endpoints of the matrix interval is indispensable. For one can take the map   $\phi : E_n \to S_n$ defined by $\phi (X) = X$ whenever $X \not=I$ and $\phi (I) = 2I$ and easily verify that it is an order embedding. 
Nevertheless, we will also give the general form of order embeddings of $E_n$ without assuming continuity in Theorem \ref{pomjan2}.

\begin{proof}
By Lemma \ref{muci}, there exists a real number $0<t<1$ such that $\phi$ is continuous at $tI$. 
By Corollary \ref{kadkrene}, we have $\phi(0)<\phi(tI)<\phi(I)$. 
By Corollary \ref{jebemt}, there exist continuous order isomorphisms $\psi_1: E_n\to E_n$ and $\psi_2: [\phi(0), \phi(I)]\to E_n$ such that $\psi_1(I/2) = tI$ and $\psi_2(\phi(tI)) = I/2$. 
Considering $\psi_2\circ\phi\circ \psi_1: E_n\to E_n$ instead of $\phi$, we may assume that $\phi$ is an order embedding from $E_n$ into itself, and also fixes the three points $0, I/2, I$, and is continuous at these points. 

We next prove that $\phi$ maps every projection of rank $m$ to a projection of rank $m$, $0\leq m\leq n$. 
By Corollary \ref{kadkrene}, we know that $\phi((0, I))\subset (0, I)$. 
If $P$ is a rank-one projection, then $P\nleq X$  holds for any $X\in (0, I)$. 
By the assumption that $\phi$ is continuous at $I$, 
there exists a sequence $(X_k)_{k\ge 1}$ in $\phi((0, I))\,\, (\subset (0, I))$ such that $X_k\to I$.  
It follows that $\phi(P)\nleq X_k$ for any $k\ge 1$,
hence $\phi(P)$ has $1$ as its eigenvalue. 
Since the subset $[P, I]$, which is order isomorphic to $E_{n-1}$, is mapped into $[\phi(P), \phi(I)] = [\phi(P), I]$, by Lemmas \ref{more,more} and \ref{ankar}, the multiplicity of $1$ as an eigenvalue of $\phi(P)$ is one. 
In other words, $[\phi(P), I]$ is order isomorphic to $E_{n-1}$. 
Consider the restriction $\phi|_{[P, I]}: [P, I] \to [\phi(P), I]$. 
Then the same discussion shows the following: 
For any rank-one projection $Q$ with $PQ=0$, $\phi(P+Q)$ has $1$ as an eigenvalue with multiplicity $2$. 
Iterate the same, then we have: 
For any $0\leq m\leq n$ and any rank-$m$ projection $P$, $\phi(P)$ has $1$ as an eigenvalue with multiplicity $m$. 
Let $0\leq m\leq n$. 
Using the continuity of $\phi$ at $0$, a similar argument ensures that if $P$ is a projection with rank $m$, then $\phi(P)$ has $0$ as an eigenvalue with multiplicity $n-m$. 
It follows that  if $P$ is a projection with rank $m$ then so is $\phi(P)$. 
By the assumption that $\phi$ fixes $I/2$ and is continuous at $I/2$, we also see that for any projection $P$, there exists a projection $Q$ such that $\rank P= \rank Q$ and  $\phi(P/2)=Q/2$. 
Since $\phi$ is an order embedding, we obtain $\phi(P/2)=\phi(P)/2$.

Let us show that the restriction $\phi|_{P_{n}^1}$ is an isometry with respect to the metric induced by the operator norm. For any subspace $V\subset \mathbb{C}^n$ we denote by $P^1(V) \subset P_{n}^1$ the subset of all rank-one projections whose ranges are contained in $V$.
It suffices to show that $\phi|_{P^1(V)}$ is an isometry for any two-dimensional subspace $V\subset \mathbb{C}^n$. 
Let $P\in P^1(V)$. 
Take the unique projection $Q\in P^1(V)$ such that $PQ=0$. 
By the preceding paragraph, $\phi(P+Q)$ is a projection of rank two. 
Denote by $W$ the range of $\phi(P+Q)$, and take the unique projection $Q'\in P^1(W)$ such that $\phi(P)Q'=0$. 
For  a real number $0\leq c\leq 1$, we have $P\leq P+cQ\leq P+Q$, and therefore, $\phi(P) \leq \phi(P+cQ)\leq \phi(P)+ Q'$. 
Thus there exists a monotone increasing function $f_P: [0, 1]\to [0, 1]$ such that $\phi(P+cQ) = \phi(P) + f_P(c) Q'$. 
We prove that $f_P(c)\to 0$ as $c\to 0$. 
By the continuity of $\phi$ at $0$, we know that $\phi(2c(P+Q))\to 0$ as $c\to 0$. 
However, by $P+cQ\ngeq 2c(P+Q)$, $c >0$, we obtain $\phi(P)+ f_P(c) Q'=\phi(P+cQ)\ngeq \phi(2c(P+Q))$.
Hence we conclude that $f_P(c)\to 0$ as $c\to 0$.

Since $\phi(E/2)= \phi(E)/2$ for any $E\in P^1(V)$, Lemma \ref{coron} with $d=1/2$ assures the continuity of $\phi|_{P^1(V)}$ at $P$. 
Hence $\phi|_{P^1(V)}$ is an injective continuous map from $P^1(V)$ into $P^1(W)$. 
Since both $P^1(V)$ and $P^1(W)$ can be identified with the 2-dimensional compact connected manifold ${\cal S}$ without boundary, by the invariance of domain theorem, we actually have $\phi(P^1(V)) = P^1(W)$. 
Fix $P\in P^1(V)$, $Q$ as above again. 
It follows that for any $0< c< 1/2$, $\phi$ maps the collection 
$\{E\in P^1(V)\, : \, E/2 \leq P+dQ \,\,\text{ for any }\,\, d>c,\,\, E/2 \nleq P+dQ\,\, \text{ for any }\,\,d<c\}$ into the collection $\{E\in P^1(W)\, : \, E/2\leq \phi(P) + f_P(d)Q'\,\, \text{ for any }\,\,d>c,\,\, E/2\nleq\phi(P) + f_P(d)Q' \,\,\text{ for any }\,\,d<c\}$. 
By Lemma \ref{coron}, the former collection is equal to $\{E\in P^1(V)\, : \, \lVert P-E\rVert = \sqrt{c/(1-c)}\}$ and hence homeomorphic to the circle. 
The latter collection is 
$$
\{E\in P^1(W) : \lVert P-E\rVert \in  [\lim_{d\uparrow c} \sqrt{f_P(d)/(1-f_P(d))}, \lim_{d\downarrow c} \sqrt{f_P(d)/(1-f_P(d))}]\}.
$$
Since $\phi$ restricts to a homeomorphism from $P^1(V)$ onto $P^1(W)$, we see that $f_P$ is a continuous monotone increasing bijection on $[0, 1/2]$. 
Hence we may apply Lemma \ref{tvvadba} for 
$$\psi=\phi|_{P^1(V)}: P^1(V)\to P^1(W),
$$ 
and
$$g_P(\sqrt{c/(1-c)})=\sqrt{f_P(c)/(1-f_P(c))},  \ \ \ 0\leq c\leq 1/2, 
$$
to obtain that $\phi|_{P^1(V)}$ is an isometry.  

Since $P_{n}^1$ is a compact connected manifold without boundary, we see that $\phi$ restricts to a surjective isometry on $P_{n}^1$ by the invariance of domain theorem. 
Applying Wigner's unitary-antiunitary theorem, we see that there exists a unitary or antiunitary $u: \mathbb{C}^n \to \mathbb{C}^n$ such that $\phi(P) = uPu^*$, $P\in P_{n}^1$. 
We will prove that $\phi(A) = uAu^*$ for any $A\in E_n$. 
It suffices to show that the map $\phi_0: E_n\to E_n$ given by $\phi_0(A) := u^*\phi(A) u$ is the identity map on $E_n$. 
It is easy to see that $\phi_0$ fixes every (not necessarily rank-one) projection. 

Take any two-dimensional subspace $V\subset \mathbb{C}^n$ and mutually orthogonal projections $P, Q\in P^1(V)$. 
Fix a real number $0\leq c<1/2$. 
By Lemma \ref{coron} with $d=1/2$, for $E\in P^1(V)$, the condition $\lVert P-E\rVert \leq \sqrt{c/(1-c)}$ is equivalent to $E/2\leq P+cQ$, which is in turn equivalent to $\phi_0(E/2)\leq \phi_0(P+cQ)$. 
We know that $\phi_0(E/2)=E/2$. 
Moreover, since $P=\phi_0(P)\leq \phi_0(P+cQ)\leq\phi_0(P+Q)=P+Q$, there exists a real number $c'\in [0, 1]$ such that $\phi_0(P+cQ)= P+c'Q$. 
Hence, by Lemma \ref{coron} again, the condition $\phi_0(E/2)\leq \phi_0(P+cQ)$ is equivalent to $\lVert P-E\rVert \leq \sqrt{c'/(1-c')}$. 
Therefore, for $E\in P^1(V)$, two conditions $\lVert P-E\rVert \leq \sqrt{c/(1-c)}$ and $\lVert P-E\rVert \leq \sqrt{c'/(1-c')}$ are equivalent. 
We obtain $c=c'$, that is, $\phi_0(P+cQ)=P+cQ$ for any $0\leq c<1/2$. 

We fix a two-dimensional subspace $V\subset \mathbb{C}^n$ again, but fix $E\in P^1(V)$ instead of $P$. 
Let $0<d< 1$ be a real number. 
By Lemma \ref{coron} with $c=d/2$, for $P\in P^1(V)$, the condition $\lVert P-E\rVert \leq \sqrt{(1-d)/(2-d)}$ is equivalent to $dE\leq P+dQ/2$, where $Q$ is the unique projection in $P^1(V)$ with $PQ=0$. 
Then this condition is in turn equivalent to $\phi_0(dE)\leq \phi_0(P+dQ/2)$. 
The preceding paragraph implies $\phi_0(P+dQ/2)=P+dQ/2$. 
Since $0 =\phi_0(0)\leq \phi_0(dE)\leq \phi_0(E)=E$, there exists $d'\in (0, 1)$ such that $\phi_0(dE)=d'E$. 
Hence, by Lemma \ref{coron} again, the condition $\phi_0(dE)\leq \phi_0(P+dQ/2)$ is equivalent to $\lVert P-E\rVert \leq \sqrt{d(1-d')/(d' (2-d))}$. 
It follows that $(1-d)/(2-d) = d(1-d')/(d' (2-d))$, and consequently,
$d=d'$. 

We have shown that $\phi_0(dE)=dE$ for any $E\in P_{n}^1$ and $d\in [0, 1]$. 
Let $A\in E_n$. 
Lemma \ref{naokoli} implies that  the ranges of $A$ and $\phi_0 (A)$ coincide and that ${\rm tr}\, (A^\dagger P) = {\rm tr}\, (\phi_0 (A)^\dagger P)$ for every $P \in P_{n}^1$ whose range is contained in the range of $A$. It follows that $A^\dagger  = \phi_0 (A)^\dagger$, and therefore $A = \phi_0 (A)$.
\end{proof}

\begin{theorem}\label{autcon2}
Let $U \subset S_n$ be a matrix domain and $\phi : U \to S_n$ an order embedding. Then $\phi (U) \subset S_n$ is open and $\phi : U \to \phi (U)$ is a homeomorphism.
\end{theorem}

\begin{proof}
By the invariance of domain theorem, it suffices to show continuity. 
Let $A\in U$. 
Lemma \ref{muci} implies that we can take real numbers $c_0, c_1>0$ such that $[A-c_0 I, A+c_1 I]\subset U$ and $\phi$ is continuous at the two points $A-c_0 I, A+c_1 I$. By Corollary \ref{kadkrene}, $\phi(A-c_0 I) < \phi(A+c_1 I)$.
Using the preceding theorem, we see that $\phi$ restricts to an order isomorphism from $[A-c_0 I, A+c_1 I]$ onto $[\phi(A-c_0 I), \phi(A+c_1 I)]$. 
We know that every order isomorphism between two operator intervals is continuous, and hence, $\phi$ is continuous at $A$.
\end{proof}

Theorem \ref{coffee} is an easy consequence of Theorem \ref{autcon2} and results in the preceding sections. Indeed, assume first that a map $\phi : U \to S_n$ has a continuous extension to $U \cup \Pi_n$ that maps $\Pi_n$ biholomorphically onto itself. Then, by Theorem \ref{punodosta},  $\phi : U \to S_n$ is a local order isomorphism. It follows from Proposition \ref{rdrit} and Lemma \ref{finite} that the maximal extension  of $\phi$ is an order isomorphism which trivially yields that $\phi$ is an order embedding. The other direction is even easier. If $\phi$ is an order embedding, then by the previous theorem $\phi : U \to \phi (U)$ is an order isomorphism between two matrix domains. In particular, it is a local order isomorphism, and the desired conclusion follows again from Theorem \ref{punodosta}.\medskip

It is now clear that if we want to have a full understanding of order embeddings of matrix domains we only need to understand the structure of maximal order isomorphisms of matrix domains. On one hand, one can say that this has been already achieved in the previous sections, but on the other hand, one would like to have concrete matrix formulae for such maps. In order to obtain such formulae we need some more notation. If $0 \le p \le n$, then the symbol $S_n(p)$ will stand for the set of all invertible $n \times n$ hermitian matrices with exactly $p$ positive eigenvalues  and $n-p$ negative eigenvalues. Here, each eigenvalue is counted with its multiplicity.
Let $m,p$ be nonnegative integers with $p \le m \le n$. We write each matrix $X \in S_n$ as a block matrix
$$
X = \left[ \begin{matrix}  X_{11} & X_{12} \cr X_{12}^\ast & X_{22} \cr  \end{matrix} \right] ,
$$
where $X_{11} \in S_m$. Put 
$$
U(m,p) = \{ X \in S_n \, : \, X_{11} \in S_m (p) \}.
$$
Clearly, $U(m, p)$ is a matrix domain in $S_n$. Define a map $\phi_{m,p} : U(m,p) \to U(m, m-p)$ by
$$
\phi_{m,p} (X) = \left[ \begin{matrix}  -X_{11}^{-1} &  iX_{11}^{-1} X_{12} \cr -iX_{12}^\ast  X_{11}^{-1}  & X_{22} -  X_{12}^\ast  X_{11}^{-1}X_ {12}\cr  \end{matrix} \right] , \ \ \ 
X = \left[ \begin{matrix}  X_{11} & X_{12} \cr X_{12}^\ast & X_{22} \cr  \end{matrix} \right] \in U(m,p).
$$
We need to explain that in the extreme case when $m=p=0$ we have $X = X_{22}$ and in this case $\phi_{0,0}$ is the identity map defined on the whole $S_n$. 
In the other extreme case when $m=n$ the map $\phi_{n,p}$ is $X \mapsto - X^{-1}$, mapping $S_n (p)$ onto $S_n (n-p)$.

A straightforward calculation shows that 
for any $X = \left[ \begin{matrix}  X_{11} & X_{12} \cr X_{12}^\ast & X_{22} \cr  \end{matrix} \right] \in U(m,p)$,  we have 
\begin{equation}\label{3by3}
-\left[ \begin{matrix} X_{11} & X_{12} & 0 \cr X_{12}^\ast & X_{22} & iI \cr 0 & -iI & 0 \cr \end{matrix} \right]^{-1} =  
 \left[ \begin{matrix}  -X_{11}^{-1} & 0 &  iX_{11}^{-1} X_{12} \cr 0 & 0 & -iI \cr -iX_{12}^\ast  X_{11}^{-1}  & iI & X_{22} -  X_{12}^\ast  X_{11}^{-1}X_ {12}\cr  \end{matrix} \right]
\end{equation}
in $M_{2n-m}$, where $I$ stands for the $(n-m) \times (n-m)$ identity matrix.
Obviously, $t \mapsto \left[ \begin{matrix} X_{11} & tX_{12} & 0 \cr tX_{12}^\ast & tX_{22} & iI \cr 0 & -iI & 0 \cr \end{matrix} \right]$, $t \in [0,1]$, is a continuous path consisting of invertible hermitian matrices connecting 
$\left[ \begin{matrix} X_{11} & 0 & 0 \cr 0 & 0 & iI \cr 0 & -iI & 0 \cr \end{matrix} \right]$ with $\left[ \begin{matrix} X_{11} & X_{12} & 0 \cr X_{12}^\ast & X_{22} & iI \cr 0 & -iI & 0 \cr \end{matrix} \right]$,
and therefore,
$\left[ \begin{matrix} X_{11} & X_{12} & 0 \cr X_{12}^\ast & X_{22} & iI \cr 0 & -iI & 0 \cr \end{matrix} \right]\in S_{2n-m}(n+p-m)$.

\begin{lemma}\label{phimp}
Let $0 \le p \le m \le n$.
The map $\phi_{m, p}$ is a maximal order isomorphism from $U(m, p)$ onto $U(m, m-p)$.
\end{lemma}
\begin{proof}
A direct calculation shows that $\phi_{m, p}$ is a bijection from $U(m, p)$ onto $U(m, m-p)$, and the inverse map is $\phi_{m, m-p}$. 
Since $S_{2n-m}(n+p-m)\subset S_{2n-m}$ is a matrix domain, 
Theorem \ref{coffee} implies that the map $Y\mapsto -Y^{-1}$ is a maximal order isomorphism from $S_{2n-m}(n+p-m)$ onto $S_{2n-m}(n-p)$. 
Hence the equation (\ref{3by3}) shows that $\phi_{m, p}$ is an order preserving map. 
Similarly, $\phi_{m, m-p}$ also preserves order, so $\phi_{m, p}$ is an order isomorphism. 
The maximality of $\phi_{m, p}$ is clear.
\end{proof}

\begin{lemma}\label{upper}
Let $0 \le p \le m \le n$.
Take the maximal integer $k$ such that there exist $X\in U(m, p)$, $Y\in S_n$ with $Y\ge 0$, $\rank Y=k$ and $X+cY\in U(m, p)$ for all real numbers $c\ge 0$. 
Then $k=n+p-m$. 
Similarly, take the maximal integer $l$ such that there exist $X\in U(m, p)$, $Y\in S_n$ with $Y\ge 0$, $\rank Y=l$ and $X-cY\in U(m, p)$ for all $c\ge 0$. 
Then $l=n-p$. 
\end{lemma}
\begin{proof}
We prove the first half. 
The second half can be proved in a similar manner.
Fix $X = \left[ \begin{matrix}  X_{11} & X_{12} \cr X_{12}^\ast & X_{22} \cr  \end{matrix} \right] \in U(m,p)$. 
Putting $Y = \left[ \begin{matrix} (X_{11})_+ & 0 \cr 0 & I \cr  \end{matrix} \right] \in S_n$, we obtain $k\ge n+p-m$. 
Let $Y = \left[ \begin{matrix}  Y_{11} & Y_{12} \cr Y_{12}^\ast & Y_{22} \cr  \end{matrix} \right] \in S_n$, where we use the same block decomposition as $X$, satisfy $Y\ge 0$ and $\rank Y\ge n+p-m+1$. 
Then we have $Y_{11}\ge 0$ and $\rank Y_{11}\ge p+1$, which implies $X+cY\notin U(m, p)$ for some $c\ge 0$.
\end{proof}

\begin{theorem}\label{embed}
Two maximal order isomorphisms $\phi_{m,p}$ and $\phi_{m',p'}$ are equivalent if and only if $m= m'$ and $p=p'$. 
If $U \subset S_n$ is a matrix domain and $\phi : U \to S_n$ is an order embedding then there exists a unique pair of nonnegative integers $p,m$ with $p \le m \le n$ such that $\phi$ is equivalent to a restriction of $\phi_{m,p}$.
\end{theorem}
\begin{proof}
Suppose that  $\phi_{m,p}$ and $\phi_{m',p'}$ are equivalent.
Then there exists an order automorphism of $S_n$ that maps $U(m, p)$ onto $U(m', p')$. 
By Lemma \ref{upper}, we have $n+p-m=n+p'-m'$ and $n-p=n-p'$, hence $m=m'$ and $p=p'$. 
Thus we have $(n+2)(n+1)/2$ equivalence classes of maximal order isomorphisms $\{ \phi_{m, p}\}_{0 \le p \le m \le n}$.
By Propositions \ref{rdrit}, \ref{UAUB}, Corollary \ref{number} and Theorem \ref{autcon2}, we conclude that each order embedding of a matrix domain is equivalent to a restriction of $\phi_{m,p}$ for a unique pair of integers $m,p$, $0 \le p \le m \le n$.
\end{proof}

\begin{remark}
In fact, we may prove that $\phi_{m, p}$ is equivalent to $\Phi_A$ for any $A\in S_n$ with $\rank A_+ = p$ and $\rank A_- = m-p$. The proof is left to the reader.
\end{remark}

Theorem  \ref{embed} has an interesting consequence. 
By the equation (\ref{3by3}), every order embedding is in a sense a ``corner'' of the order isomorphism $X \mapsto -X^{-1}$. 
Indeed, let $0 \le p \le m \le n$, and let $\psi, \psi' : S_n \to S_{2n-m}$ be affine order embeddings defined by 
$$
\psi\left(\left[ \begin{matrix}  X_{11} & X_{12} \cr X_{12}^\ast & X_{22} \cr  \end{matrix} \right]\right)= \left[ \begin{matrix} X_{11} & X_{12} & 0 \cr X_{12}^\ast & X_{22} & iI \cr 0 & -iI & 0 \cr \end{matrix} \right],\quad 
\psi'\left(\left[ \begin{matrix}  X_{11} & X_{12} \cr X_{12}^\ast & X_{22} \cr  \end{matrix} \right]\right)= 
\left[ \begin{matrix}  X_{11} & 0 &  X_{12} \cr 0 & 0 & -iI \cr X_{12}^\ast  & iI & X_{22}\cr  \end{matrix} \right]. 
$$
Then the equation (\ref{3by3}) means $\phi_{m, p}(X) = \psi'^{-1}(-\psi(X)^{-1})$, $X\in U(m, p)$. 
Therefore, for any order embedding $\phi: U\to S_n$ of a matrix domain $U\subset S_n$ there exist $0\le m\le n$ and affine order embeddings $\psi_1, \psi_2 : S_n\to S_{2n-m}$ such that $\phi(X)=\psi_2^{-1}(-\psi_1(X)^{-1})$, $X\in U$.

\section{Applications, Examples and Final Remarks}\label{apply}
In the first part of this section we compare our results with the classical Loewner's theorem. 
To begin with, let us give the precise statement of the classical Loewner's theorem. 

\begin{theorem}[Loewner]\label{loewner}
Let $f: (a, b)\to \R$ be a non-constant function on an open interval. The following are equivalent:
\begin{enumerate}
\item $f$ is operator monotone.
\item $f$ has an analytic continuation to the upper half-plane $\Pi$ which maps $\Pi$ into $\Pi$.
\item There exist a finite measure $\mu$ on $\R\setminus (a, b)$ and constants $c\in \R$, $d\ge0$ such that
\begin{equation}\label{pick}
f(x)= c+dx + \int_{\R\setminus (a, b)}\frac{1+xy}{y-x}\, d\mu(y).
\end{equation}
\end{enumerate}
\end{theorem}
We remark that the equivalence (2)$\Leftrightarrow$(3) is a classical fact of complex analysis that was known before Loewner's theorem. 
It is not difficult to show the implication (3)$\Rightarrow$(1) directly, but let us revisit it through the scope of our Theorem \ref{1}. 
\begin{proof}[Proof of (3)$\Rightarrow$(1) of Theorem \ref{loewner} based on Theorem \ref{1}]
Let $f$ be defined by (\ref{pick}). 
We may think of $f$ as a holomorphic function on $(\mathbb{C} \setminus \mathbb{R}) \cup (a,b)$. 
Hence  $f$ determines a map from $\cU:=\{X\in B(H)\, :\, \sigma(X)\subset (\mathbb{C} \setminus \mathbb{R}) \cup (a,b)\}$ into $B(H)$ by holomorphic functional calculus. 
By standard arguments $\cU$ is open and $f$ is a holomorphic map from $\cU$ into $B(H)$.
Note that $\Pi(H)\subset \cU$. 
Let $X\in \Pi(H)$. Then 
$$
f(X)= cI+dX + \int_{\R\setminus (a, b)}g_X(y) \, d\mu(y), 
$$
where $g_X(y) := (y^2+1)(yI-X)^{-1} -yI$, $y\in \R\setminus (a, b)$. 
It is easy to see that $g_X(y)\in \Pi(H)$, $y\in \R\setminus (a, b)$. 
Since $f$ is non-constant, we have either $d>0$ or $\mu(\R\setminus(a, b))>0$, which implies $f(X)\in \Pi(H)$.
By Theorem \ref{1}, $f$ determines an order preserving map on $(aI, bI)$, that is, $f$ is operator monotone.
\end{proof}

Therefore, taking the equivalence (2)$\Leftrightarrow$(3) as granted, we may think of the direction from holomorphic maps to order preserving maps of the classical and our Loewner's theorem in a unified simple manner. 
The essential part of Loewner's achievement is the implication (1)$\Rightarrow$(2), or from order preserving maps to holomorphic maps. 
Simon calls it the ``hard direction'', and gives various proofs in \cite{Sim}. 
In our result the corresponding direction is obtained with a relatively short proof in Section \ref{hard}. 
However, this is not at all easy: we completely rely on the formula (\ref{qaui})  in Theorem \ref{fpfq} by the second author whose proof is quite involved.

In the paper we have treated order embeddings of matrix domains and local order isomorphisms of operator domains, that is, maps that preserve order in both directions and maps that are locally bijective and locally preserve order in both directions, respectively. 
Having the classical Loewner's theorem in mind it is natural to ask why we did not consider maps preserving order in one direction only. The reason is that such maps may have wild behaviour even under additional  assumptions. 
For example, let us recall the fixed-dimensional version of Loewner's theorem: 
\begin{theorem}[Loewner]\label{nloewner}
Let $2\le n<\infty$ be an integer and $f: (a, b)\to \R$ a map on an open interval. The following are equivalent:
\begin{enumerate}
\item $f$ is a matrix monotone function of order $n$, that is, $f$ determines an order preserving map from the matrix interval $(aI, bI)\subset S_n$ into $S_n$.
\item $f$ is continuously differentiable, and for any $n$-tuple $a<\lambda_1<\cdots<\lambda_n<b$, the Loewner matrix $(f^{[1]}(\lambda_i, \lambda_j))_{i,j}\in S_n$ is positive semidefinite, where 
$$
f^{[1]}(\lambda_i, \lambda_j) =  \left\{ \begin{matrix} \frac{f(\lambda_j)-f(\lambda_i)}{\lambda_j-\lambda_i} & {\rm if} & i\neq j \cr 
f'(\lambda_i) & {\rm if} & i=j. \cr \end{matrix} \right.
$$
\end{enumerate}
\end{theorem}
Even though this theorem gives an explicit and complete characterization of order preserving maps determined by functional calculus, this class of maps is still not well understood. 
See Chapter 14 of \cite{Sim}, in which a succinct description is explained only in the case $n=2$.

For another class of wild examples of maps that preserve order in one direction only, take the operator domain $(0,I) \subset S(H)$. 
Let $f : (0,I] \to (0, 1]\subset \mathbb{R}$ be a continuous function. 
Assume that for every $X,Y \in (0,I)$ we have $X \le Y \Rightarrow f(X) \le f(Y)$.
Suppose further that $f(X)=1$ for every $X \in (0,I]\setminus (0, I)$.
Then the map $\phi : (0,I) \to (0,I)$ given by
$\phi (X) = f(X)X^{1/2}$ is a continuous bijective map preserving order in one direction. Indeed, since the function $g(x) = x^{1/2}$ is operator monotone on $(0,1)$ it is trivial to verify that $\phi$ is continuous and preserves order in one direction. Define a map $\psi : (0,I) \to (0,I)$ by $\psi (X) = (f(X))^2 X = (\phi (X))^2$, $X\in (0,I)$. To show the bijectivity one only needs to see that for every $X \in (0,I)$ the set $ \{ sX \, : s \in \mathbb{R} \} \cap (0,I)$ is mapped by $\psi$ bijectively onto itself. 
One can find a lot of rather ``wild" examples of functions $f : (0,I] \to (0, 1]$ satisfying the above conditions. 
Moreover, we may compose arbitrary order automorphisms of $(0, I)$ from both sides of $\phi$ to obtain a wider class of examples, showing that there is no much hope to get any reasonable structural result for continuous bijections of the operator domain $(0,I)$ onto itself preserving order in one direction only.

Let us give here just one example of a function $f$ with the condition above. 
Let $n$ be a positive integer and $f_1 , \ldots , f_n : (0,1] \to (0,1]$ any surjective increasing functions. For every $x \in H$ of norm one we choose an integer $k$, $1 \le k \le n$, and set $f_x = f_k$. Define $f : (0,I] \to (0, 1]$ by
$$
f(X) = \sup_{x\in H, \| x \| = 1} \{ f_x ( \langle Xx,x \rangle ) \}.
$$
It is easy to see that $f$ satisfies the desired condition.\medskip

In the finite-dimensional case we have a complete understanding of order embeddings $\phi : U \to S_n$ for every matrix domain $U \subset S_n$. We know that there are no order embeddings 
$\phi : U \to S_n$ when $U$ is a matrix domain in $S_m$ with $m > n$. We will next show that there is almost no hope to get any reasonable structural result for order embeddings when $m < n$. 
If a map $\phi : U \to S_n$ is defined by
$$
\phi (X) = \left[ \begin{matrix}  X & 0  \cr 0 & \varphi (X) \cr  \end{matrix} \right]
$$
where $\varphi : U \to S_{n-m}$ is any map preserving order in one direction, then clearly, $\phi$ is an order embedding. 
However, we already know that maps preserving order in one direction may be wild even under additional assumptions, and even under the linearity assumption, see \cite[Section 8]{Sto}. 
Thus, there is no nice description of order embeddings of operator domains in $S_m$ into $S_n$ when $n > m$. 
Let us just add here a remark that $\varphi$ may have the special form $\varphi(X)=\varphi_0(x_{11}) A$, $X = (x_{ij})_{1\le i, j\le m}$, where $\varphi_0 : \mathbb{R} \to \mathbb{R}$ can be any increasing real function, possibly non-continuous, and possibly constant on some intervals, and $A\in S_{n-m}$ can be any positive semi-definite matrix.

Since each infinite-dimensional Hilbert space $H$ can be identified with an orthogonal direct sum of two copies of $H$, the same idea can be used to construct wild examples of order embeddings of $S(H)$ into
$S(H) \cong S(H \oplus H)$. To make matters worse, there are other examples of wild order embeddings of $S(H)$ into $S(H)$. 
Take the case when $H$ is separable and choose an orthonormal basis $\{e_k \, : \, k \ge 1 \}$ 
and a countable dense subset $\{h_k \, : \, k \ge 1 \}$ in the unit ball of $H$. Let further $\varphi_k : \mathbb{R} \to \mathbb{R}$, $k \ge 1$, be any sequence of strictly increasing functions satisfying the condition that for every bounded sequence $(q_k) \subset \mathbb{R}$
the sequence $(\varphi_k (q_k))$ is bounded. 
Then the map $\phi : S(H) \to S(H)$ given by $\phi (X) = D_X$, where $D_X \in S(H)$ is
the ``diagonal operator" defined by
$$
D_X e_k = \varphi_k ( \langle X h_k , h_k \rangle ) e_k, \ \ \ k=1,2,\ldots
$$
is an order embedding. All these examples show that while having a full understanding of order embeddings in the finite-dimensional case, we need to work with local order isomorphisms (the local bijectivity assumption)
when dealing with the infinite-dimensional setting.\medskip

When formulating our results it is essential that we assume that local order isomorphisms are maps defined on operator domains,
that is, open and connected subsets of $S(H)$. 
Indeed, the assumption of openness is indispensable when dealing with differentiability.  
See also Theorem \ref{pomjan2} below, in which a difficulty in dropping the assumption of openness can be observed.
In order to see that the assumption of connectedness is essential we can take
$U= (0,I)$ and $V=  (-I, 0)\cup (0,I)$. If $\psi : (-I,0) \to (-I, 0)$ is any order automorphism, then the map $\phi : V \to V$ defined by $\phi (X) = X$, $ X \in  (0,I)$, and $\phi (X) = \psi (X)$, $X \in (-I,0)$,
is an order automorphism of $V$ extending the identity automorphism of $U$.

At first look it might be surprising that we are dealing with local order isomorphisms rather than with order isomorphisms. 
First of all, the notion of local order isomorphisms appears naturally when studying an analogue of the Loewner's theorem for maps on operator domain. 
Another reason is that for order isomorphisms we do not have the unique extension property. 
Indeed, by Zorn's lemma any order isomorphism between operator domains has an extension to a maximal order isomorphism.
Suppose $A\in S(H)$ and both $A_+$ and $A_-$ are noncompact. 
We show that there exists a domain $U\subset U_A$ such that $\Phi_A$ restricts to an order isomorphism from $U$ onto $\Phi_A(U)$ but this order isomorphism has more than one extension to maximal order isomorphisms. 
We may take a domain $V\subset U_A$ such that $\Phi_A$ restricts to a maximal order isomorphism from $V$ onto $\Phi_A(V)$. 
Note that $V\neq U_A$ by Theorem \ref{compact}.
Take an operator $X\in U_A\cap \partial V \,\,(\neq \emptyset)$. 
We may take a domain $X\in W\subset U_A$ such that $\Phi_A$ restricts to a maximal order isomorphism from $W$ onto $\Phi_A(W)$. 
Hence $\Phi_A$ restricted to $U:=V\cap  W$ has two extensions to maximal order isomorphisms.
See also \cite{AMY}, in which a local condition is essential in giving a variant of Loewner's theorem, too.\medskip

Our results unify and substantially improve all known results on order isomorphisms of operator intervals. 
We set $(-\infty , \infty):=S(H)$.
Take any collection $J$ of one of the forms  
$$
[A,B], \ \,  (A, B), \ \,  [A, B),  \ \,  (A ,B], \ \, (A, \infty), \ \,  [A, \infty),  \ \,  (- \infty , A], \ \, (- \infty , A), \ \, (-\infty, \infty)
$$
for some $A, B\in S(H)$ with $A<B$. 
In \cite{Se6} an explicit construction 
of an order isomorphism from $J$ onto one of the five operator intervals below was constructed:
\begin{equation}\label{mama}
[0,I], \ \,  (0, \infty), \ \,  [0, \infty),  \ \,  (- \infty , 0], \ \,  (-\infty, \infty).
\end{equation}
Moreover, it was shown that any two of (\ref{mama}) are not order isomorphic. 

Thus, the basic theorems on order isomorphisms of operator intervals describe the general forms of order automorphisms of operator intervals (\ref{mama}). 
We already described the case $(-\infty, \infty)$ in Theorem \ref{molnar}. 
It has been known \cite{Mo1, Se6} that for each order automorphism $\phi : J \to J$, where $J$ is any of the intervals 
$$
(0, \infty), \ \,  [0, \infty),  \ \,  (- \infty , 0],  
$$
there exists a bounded bijective linear or conjugate-linear operator $T: H \to H$  such that
\begin{equation}\label{veter}
\phi (X) = TXT^\ast 
\end{equation}
for every $X \in J$. 
Let us give a proof based on our results. 
Suppose that $\phi$ is an order automorphism of $(0, \infty)$. 
Then the map $\psi: (-I, \infty)\to (-\phi(I), \infty)$ defined by $\psi(X)=\phi(X+I)-\phi(I)$ is an order isomorphism with $\psi(0)=0$.
It follows by Theorem \ref{TPhi} that there exist $A\in S(H)$ and an invertible bounded linear or conjugate-linear operator $T: H\to H$ such that  $(-I, \infty)\subset U_A$ and 
$$
\psi(X) = T\Phi_A(X)T^\ast =  T(XA+I)^{-1}XT^\ast,\quad X\in (-I, \infty).
$$
The inclusion $(-I, \infty)\subset U_A$ implies $A\ge 0$. 
Since $T^{-1}Y(T^\ast)^{-1}\in U_{-A}$ for all $Y\in (-\phi(I), \infty)$, we also obtain $A\le 0$, hence $A=0$.
This leads to the formula (\ref{veter}) if $J=(0, \infty)$.
Suppose that $\phi$ is an order automorphism of $[0, \infty)$. 
It is not difficult to see that $\phi$ restricts to an order automorphism of $(0, \infty)$ (imitate the second paragraph of the proof of Lemma \ref{[]}), and from this it is an easy exercise to obtain the desired conclusion when $J=[0, \infty)$.
Since $(-\infty, 0]$ is order anti-isomorphic to $[0, \infty)$ by the map $X\mapsto -X$, we also obtain the general form of order automorphisms of $(-\infty, 0]$. 

Therefore, for four of the intervals (\ref{mama}) the description of the general forms of order automorphisms is easy: all order automorphisms are congruences $X \mapsto TXT^\ast$ or congruences plus translations.  
However, in the case of the interval $[0,I]$ we have quite a complicated formula (\ref{qaui}) in Theorem \ref{fpfq}. 
This description is not really satisfactory. If we denote the map $\phi$ appearing in (\ref{qaui}) by $\phi_{p,q,T}$ then the family of all order automorphisms of the effect algebra $[0,I]$ is parametrized by three parameters $p,q,T$.
In this description there are too many parameters in the sense that we may have $\phi_{p,q,T} = \phi_{p',q',T'}$ when $p\not=p'$, $q\not=q'$, and $T \not= T'$. Using our results we present a better description of order automorphisms
of the effect algebra. 

\begin{theorem}\label{pomjan} 
Assume that $\phi : [0, I] \to [0, I]$ is an order automorphism.
Then there exists a bijective linear or conjugate-linear bounded operator
$T: H\to H$ that is unique up to a multiplication with a complex number of modulus one, such that 
$$
\phi (X) = T\Phi_{T^\ast T -I}(X)T^\ast = T \left( X (T^\ast T -I) +I \right)^{-1} X T^\ast
$$
for every
$X \in [0,I]$.
Conversely, for any bijective linear or conjugate-linear bounded operator $T\in B(H)$, we have $[0, I]\subset U_{T^\ast T -I}$ and $\phi_T: X\mapsto T\Phi_{T^\ast T -I}(X)T^\ast$ is an order automorphism of $[0, I]$. 
\end{theorem}
\begin{proof}
It follows from Lemma \ref{ABCD} that $\phi$ extends to a biholomorphic map that maps $\Pi(H)$ onto itself.
By Proposition \ref{harris}, there exist $A\in S(H)$ and a linear bounded bijection $T\in B(H)$ with either
\[
\begin{split}
\phi(X) &= T(X^{-1}+A)^{-1}T^\ast = T(XA+I)^{-1}XT^\ast, \ \ \ X\in \Pi(H), \text{ or}\\
\phi(X) &= T((X^t)^{-1}+A)^{-1}T^\ast = T(X^t A+I)^{-1}X^t T^\ast, \ \ \ X\in \Pi(H).
\end{split}
\] 
By Lemma \ref{diverge}, we have $[0, I]\subset U_A$, and the continuity of $\phi$ implies that either  $\phi(X) = T\Phi_A(X)T^\ast$, $X\in [0, I]$, or  $\phi(X) = T(\Phi_{A^t}(X))^t T^\ast$, $X\in [0, I]$.
Since $\phi(I)=I$, we obtain $I=T(A+I)^{-1}T^\ast$, hence $A=T^\ast T-I$. 
The uniqueness of $T$ up to a multiplication with a complex number of modulus one is a consequence of Proposition \ref{AA'}.

Conversely, let $T : H \to H$ be a bijective linear or conjugate-linear bounded operator. 
Since $(tI) (T^\ast T - I) +I$ is invertible for every $t \in [0,1]$ we have $I \in U_{T^\ast T -I}$. 
Because $T^\ast T-I>-I$, Proposition \ref{interval} implies $[0, I]\subset U_{T^\ast T -I}$, and $\phi_T$ is an order isomorphism from $[0, I]$ onto $[0, \phi_T(I)]=[0, I]$.   
\end{proof}

In the finite-dimensional case we have a stronger result.
\begin{theorem}\label{pomjan2} 
Assume that $\phi : E_n \to S_n$, $n \ge 2$, is an order embedding.
Then there exist an invertible $n \times n$ matrix $T$ and $A, B \in S_n$ with $A > -I$ such that either
$$
\phi (X) =  T \left( X A +I \right)^{-1} X T^\ast +B, \ \ \
X \in E_n \setminus \{ 0,I \},
$$
or
$$
\phi (X) =  T \left( X^t A  +I \right)^{-1} X^t  T^\ast +B, \ \ \
X \in E_n \setminus \{ 0,I \} .
$$
\end{theorem}
It is clear that $\phi(0)$ is a hermitian matrix satisfying $\phi(0) \le B$ and $\phi (I)$ a matrix satisfying $\phi (I) \ge  T \left(  A +I \right)^{-1}  T^\ast +B$.
Conversely, if $\phi$ is defined on the set $[0,I]\setminus \{ 0,I \}$ by one of the above two formulae and if $\phi(0)$ and $\phi (I)$ satisfy the previous two inequalities, then  $\phi : E_n \to S_n$ is an order embedding.

\begin{lemma}
Assume that $\phi : E_n \to S_n$, $n \ge 2$, is an order embedding. Suppose that $\phi (X) = X$ for every $X \in (0,I)$. Then $\phi (X) = X$ for every $X \in E_n \setminus \{ 0,I \}$.
\end{lemma}

\begin{proof}
It is enough to prove that for every projection $P$ of rank one and every real $t \in (0,1)$ we have $\phi(tP)= tP$. Indeed, if we can prove this, then by considering the order embedding $X \mapsto I - \phi (I -X)$ we see that $\phi (I - tP) = I-tP$, $t \in (0,1)$. Assume that we already know this. Take any $X \in E_n$. If $Xx = tx$ for some nonzero $x \in \mathbb{C}^n$ and some real $t$, $t \in (0,1)$, then $tP \le X \le I - (1-t)P$, where $P$ is the projection of rank one whose range is spanned by $x$. Therefore, $tP \le \phi (X) \le I - (1-t)P$, yielding that $\phi (X) x = tx$. If $Xx = x$ for some nonzero vector $x$, then $X \ge tP$, 
where $P$ is the projection of rank one whose range is spanned by $x$ and $t$ is any real number, $0 < t < 1$. Thus, if $X$ has eigenvalue $1$ but $0$ is not its eigenvalue and $X\not=I$, then with respect to a suitable orthonormal basis we have the following block matrix representations
$$
X = \left[ \begin{matrix} I & 0 \cr 0 & X_1 \cr \end{matrix} \right] \ \ \ {\rm and} \ \ \ \phi (X) = \left[ \begin{matrix} Y & 0 \cr 0 & X_1 \cr \end{matrix} \right],
$$
where $Y \ge I$ and $X_1$ has all its eigenvalues in the open interval $(0,1)$. If $Y \not=I$, then it is easy to find a projection $Q$ of rank one and real number $s \in (0,1)$ such that $sQ \not\le X$ but $sQ \le \phi (X)$, a contradiction. In a similar way we see that if $X\not=0$, $0$ is an eigenvalue of $X$ but $1$ is not an eigenvalue of $X$, then $\phi (X) = X$. Finally, if both $1$ and $0$ are eigenvalues of $X$, then
$$
(1 - \varepsilon)X = \phi ((1-\varepsilon)X) \le \phi (X) \le \phi ((1 - \varepsilon)X + \varepsilon I) = (1 - \varepsilon)X + \varepsilon I
$$
for any positive $\varepsilon <1$, and consequently, $\phi (X) = X$.

It remains to verify that $\phi (tP) = tP$ for every projection $P$ of rank one and every real $t \in (0,1)$. 
Assume with no loss of generality that $P = E_{11}$, the matrix whose all entries are zero but the $(1,1)$-entry that equals $1$. From $tE_{11} \le tE_{11} + \varepsilon (I - E_{11}) = \phi ( tE_{11} + \varepsilon (I - E_{11}) )$ we get $ C= \phi (tE_{11}) \le tE_{11}$.
In particular, $c_{11}$, the $(1,1)$-entry of $C$ is no larger than $t$, and
\begin{equation}\label{marij}
\langle Cx,x \rangle \le \langle tE_{11} x, x \rangle
\end{equation}
for every $x \in \mathbb{C}^n$.
It also follows that $C \le tE_{11} + (I-E_{11})$. We claim that $C \not< tE_{11} + (I-E_{11})$. For if $C < tE_{11} + (I-E_{11})$ then there would exist a positive real number $\delta<t$ such that
$$
\phi (tE_{11}) = C \le (t- \delta) E_{11} + (1 - \delta) (I - E_{11}) = \phi ( (t- \delta) E_{11} + (1 - \delta) (I - E_{11}) )
$$
yielding $tE_{11} \le (t- \delta) E_{11} + (1 - \delta) (I - E_{11})$, a contradiction. Thus, there exists a unit vector $y$ such that
$$
\langle tE_{11} y , y \rangle + \langle (I-E_{11}) y , y \rangle = \langle Cy,y \rangle.
$$
Since $I- E_{11}$ is positive we get from (\ref{marij}) that $ \langle (I-E_{11}) y , y \rangle =0$ implying that $y = be_{1}$ for some complex number $b$ of modulus one. Here, $e_1$ is the first standard basis vector. Hence $c_{11} = t$ which together with
$ C \le tE_{11}$ yields that
$$
C = \left[ \begin{matrix} t & 0 \cr 0 & C_1 \end{matrix} \right]
$$
with $C_1 \le 0$. We need to show that $C_1 = 0$. If not, then we can easily find $X \in (0,I)$ such that $X \ge C$ but $X \not\ge tE_{11}$, a contradiction.
\end{proof}

\begin{corollary}\label{patemam}
Let $A,B \in S_n$ satisfy $A < B$.
Assume that $\phi : [A,B] \to S_n$, $n \ge 2$, is an order embedding. Suppose that $\phi (X) = X$ for every $X \in (A,B)$. Then $\phi (X) = X$ for every $X \in [A,B] \setminus \{ A,B \}$.
\end{corollary}

\begin{proof} Trivial. \end{proof}

\begin{proof}[Proof of Theorem \ref{pomjan2}.]
Theorem \ref{autcon2} yields that the restriction $\phi : (0,I) \to S_n$ is a local order isomorphism. Thus, there exists a unique extension of $\phi|_{(0, I)}$ to a maximal (local) order isomorphism $\Phi: U\to S_n$. Since $ \phi (0) \le \phi (X) \le \phi (I)$ for every $X \in (0, I)$ we have $\| \phi (X) \| \le M$, $X \in (0, I)$, for some positive real number $M$. Thus Lemma \ref{diverge} implies $[0,I] \subset U$. 
By Theorem \ref{TPhi}, we conclude that there exist a bijective linear or conjugate-linear bounded operator
$T : H \to H$ and $A,B \in S(H)$ such that $[0, I]\subset U_A$ and $\Phi (X) = T \left( X A +I \right)^{-1} X T^\ast +B$ for every $X \in U$.
By Proposition \ref{interval}, we obtain $A>-I$.
We know that $\Phi ([0,I]) = [\Phi (0) , \Phi (I)]$ and that $\Phi ((0,I)) = (\Phi(0), \Phi(I))$. 

Clearly, $\psi = \phi \circ \Phi^{-1} : [\Phi (0) , \Phi (I) ] \to S_n$ is an order embedding. Moreover, $\psi (X) = X$ for every $X \in (\Phi (0) , \Phi (I) )$. It follows from Corollary \ref{patemam} that $\phi ( \Phi^{-1} (X) ) = X$ for every $X \in [\Phi (0) , \Phi (I) ] \setminus \{ \Phi (0) , \Phi (I) \}$, or equivalently, $\phi (X) = \Phi (X)$ for every $X \in [0, I] \setminus \{ 0, I \}$.
\end{proof}
\medskip\medskip

{\bf Acknowledgement.} 
The authors are grateful to \'Eric Ricard for making them aware of the reference \cite{Har}.

\end{document}